\documentclass[10pt,reqno]{amsart}
\usepackage{inputenc}
\usepackage{enumerate}
\usepackage{amssymb}
\usepackage{amsmath}
\usepackage{mathrsfs}
\usepackage{amsthm}
\usepackage{tikz-cd}
\usepackage{mathpazo}
\usepackage{verbatim}
\usepackage{extarrows}
\usepackage{color}
\usepackage{setspace}
\usepackage[colorlinks,linkcolor = blue,anchorcolor = red,urlcolor = blue,citecolor = blue]{hyperref}
\usepackage{graphicx, subfig}
\usepackage[square, comma, sort&compress, numbers]{natbib}
\usepackage{multirow}
\usepackage[titletoc]{appendix}

\theoremstyle{definition}
\newtheorem{Def}{Definition}[section]
\newtheorem{Exa}[Def]{Examples}
\newtheorem{Rem}[Def]{Remark}
\theoremstyle{plain}
\newtheorem{Thm}[Def]{Theorem}
\newtheorem{Lem}[Def]{Lemma}
\newtheorem{Pro}[Def]{Proposition}
\newtheorem{Cor}[Def]{Corollary}
\newtheorem{Que}[Def]{Question}

\usepackage[left=2.6cm, right=2.6cm, top=3cm, bottom=3cm]{geometry}
\usepackage[all]{xy}
\usepackage{bm}

\def\Box{{\bf Box}}
\def\IN{\mathbb N}\def\IR{\mathbb R}\def\IC{\mathbb C}\def\IZ{\mathbb Z}\def\IH{\mathbb H}

\def\C{\mathcal C}\def\B{\mathcal B}\def\D{\mathcal D}\def\E{\mathcal E}\def\K{\mathcal K}\def\F{\mathcal F}\def\L{\mathcal L}\def\H{\mathcal H}\def\N{\mathcal N}\def\I{\mathcal I}
\def\sB{\mathscr B}\def\sL{\mathscr L}\def\sH{\mathscr H}\def\fB{\mathit B}

\def\supp{\textup{supp}}

\def\prop{\textup{Prop}}
\def\diam{\textup{diam}}

\def\ox{\otimes}
\def\wh{\widehat}
\def\wt{\widetilde}
\def\ox{\otimes}

\def\Ga{\Gamma}

\author[L.~Guo]{Liang Guo}
\address[L.~Guo]{Shanghai Institute for Mathematics and Interdisciplinary Sciences, Shanghai, 200433, P.~R.~China}
\email{liangguo@simis.cn}

\author[Q.~Wang]{Qin Wang}
\address[Q.~Wang]{Research Center for Operator Algebras, School of Mathematical Sciences, East China Normal University, Shanghai, 200241, P.~R.~China.}
\email{qwang@math.ecnu.edu.cn}

\title{Geometric Banach property (T) for metric spaces via Banach representations of Roe algebras}
\date{\today}

\begin{document}
\maketitle

\begin{abstract}
In this paper, we introduce a notion of geometric Banach property (T) for metric spaces, which jointly generalizes Banach property (T) for groups and geometric property (T) for metric spaces. Our framework is achieved by Banach representations of Roe algebras of metric spaces. We show that geometric Banach property (T) is a coarse geometric invariant, and it is equivalent to the existence of the Kazhdan projections in the Banach-Roe algebras.  Further, we study the implications of this property for sequences of finite Cayley graphs, establishing two key results: 1. geometric Banach property (T) of such sequences implies Banach property (T) for their limit groups; 2. while the Banach coarse fixed point property implies geometric Banach property (T), the converse fails.  Additionally, we provide a geometric characterization of V.~Lafforgue's strong Banach property (T) for a residually finite group in terms of geometric Banach property (T) of its box spaces.
\end{abstract}

\tableofcontents

\section{Introduction}

As a landmark concept introduced by D.~Kazhdan in \cite{Kazhdan1967}, property (T) reveals the rigidty of unitary representations of locally compact groups on Hilbert spaces. It leads to profound applications in the fields of representation theory, geometric group theory, dynamical systems, mathematical physics, etc., with deep impacts across pure mathematics and applied mathematics. Recall that a group $\Ga$ has property (T) if every unitary representation of $\Ga$ that weakly contains the trivial representation must actually contain the trivial representation. Based on our interests, we have compiled some connections between property (T) and higher index theory:
\begin{itemize}
\item G.~Margulis \cite{Margulis1973} used property (T) to provide the first explicit construction of expander graphs, which can not be coarsely embedded into Hilbert space \cite{Gromov2003} and, furthermore, leads to counterexamples to the (maximal) coarse Baum-Connes conjecture, see \cite{WilYu2012a, WilYu2012b}.
\item The standard Dirac-dual Dirac method in studying the Baum-Connes conjecture is no longer effective for discrete groups with property (T), and the Baum-Connes conjecture with coefficients may not hold for such groups, see \cite{HLS2002}.
\end{itemize}
So far, the higher index problem for groups with property (T) remains an important unresolved issue in noncommutative geometry. One is referred to \cite{KazhdanT} for further details.

With the invention of \emph{Banach $KK$-theory} by V.~Lafforgue \cite{Lafforgue2002}, it was discovered that the Banach version of Dirac-dual-Dirac method could be adapted for certain groups with property (T) (for example, cocompact lattices of certain linear groups) within the framework of Banach $KK$-theory. However, this adaptation requires not only considering unitary representations of the group algebra but also extending the analysis to group representations on certain Banach spaces (e.g., $L^p$-spaces). In 2007, U.~Bader, A.~Furman, T.~Gelander, and N.~Monod introduced a Banach space version of property (T) for group representations on Banach spaces in \cite{BFGM2007}, and extensively discussed the distinctions and connections between this Banach variant and the classical version of property (T). They also proposed the conjecture that higher-rank algebraic groups should possess Banach property (T) with respect to the class of all super-reflexive Banach spaces. Simultaneously, V.~Lafforgue proved that the special linear group $SL_3(\mathbb F)$ over a non-Archimedean local field $\mathbb F$ has Banach property (T) with respect to all super-reflexive Banach spaces in \cite{Lafforgue2008}. This result is later generated by B.~Liao \cite{Liao2014} to all almost simple connected linear algebraic groups over a non-Archimedean field. Using this result, he constructed the first example of an expander graph that cannot be coarsely embedded into any super-reflexive Banach space, which is now referred to as a \emph{super expander}. In fact, in \cite{Lafforgue2008}, Lafforgue even introduced a stronger notion of \emph{strong Banach property (T)} and demonstrated that his innovative Banach $KK$-theory approach is no longer applicable to the Baum-Connes conjecture for groups with this stronger property. In contrast to the non-Archimedean case, the Banach property (T) for algebraic groups over $\IR$ and their lattices remained poorly understood for decades. However, after years of development, this problem was finally resolved in recent breakthroughs. In \cite{Oppen2023}, I.~Oppenheim proved that $SL_3(\IZ)$ has Banach Property (T) with respect to the class of super-reflexive Banach spaces. Very recently, this result is extended to all higher-rank algebraic groups by de Laat and de la Salle in \cite{LS2023}, thus fully resolving the conjecture posed in \cite{BFGM2007}.

Analogues to classical property (T), Banach property (T) admits serveral equivalent characterizations. In \cite{BFGM2007}, it is proved that a group has Banach property (T) with respect to a Banach space if every affine isometric action of the group on this space has a fixed point. This partially generalizes the Delorme-Guichardet theorem from Hilbert spaces to Banach spaces. In \cite{DN2019}, C.~Drutu and P.~Nowak introduced the notion of Kazhdan projection in Banach group algebras with respect to a family of uniformly convex Banach spaces, showing that a group has uniform Banach property (T) if and only if its Banach group algebra contains such a projection.

There exists a striking parallel between the development of representation theory and operator algebras theory for groups and the development of coarse geometry and operator algebra theory for metric spaces. We summarize key comparative aspects of these two frameworks in Table \ref{table: group vs coarse geometry}.  

{\begin{table}[h]
\centering
\begin{tabular}{|c|c|c|}
\hline
\multicolumn{1}{|l|}{}& Groups $\Ga$ & Metric Spaces $X$\\
\hline
\multirow{2}{*}{\begin{tabular}[c]{@{}c@{}}Operator algebra and \\ Higher index Theroy\end{tabular}} &  $C^*_r\Ga$: \emph{Group $C^*$-algebra}                                                         & $C^*(X)$: \emph{Roe algebra}\\
\cline{2-3} & Baum-Connes conjecture & Coarse Baum-Connes conjecture \\ 
\hline
\multirow{3}{*}{\begin{tabular}[c]{@{}c@{}}\\Representations \\ on \\ Hilbert spaces\end{tabular}}   & \begin{tabular}[c]{@{}c@{}}Amenablitiy\\ \cite{VonNeumann}\end{tabular} & \begin{tabular}[c]{@{}c@{}}Property A\\ \cite{Yu2000}\end{tabular} \\ \cline{2-3} 
& \begin{tabular}[c]{@{}c@{}}Haagerup property\\ \cite{Haagerup}\end{tabular} & \begin{tabular}[c]{@{}c@{}}Coarse embedding into Hilbert space\\ \cite{Gromov1993}\end{tabular} \\
\cline{2-3}  & \begin{tabular}[c]{@{}c@{}}Property (T)\\ \cite{Kazhdan1967}\end{tabular}             & \begin{tabular}[c]{@{}c@{}}Geometric property (T)\\ \cite{WY2014}\end{tabular}          \\ \hline
Banach Representations                                                                                 & \begin{tabular}[c]{@{}c@{}}Banach property (T)\\ \cite{BFGM2007}\end{tabular}         & \textbf{Geometric Banach property (T)}                                                          \\ \hline
\end{tabular}
\caption{Group vs. Coarse geometry}
\label{table: group vs coarse geometry}
\end{table}}

In \cite{WilYu2012b}, R. Willett and G. Yu first introduced the concept of \emph{geometric property (T)} as a coarse geometric analogue of Kazhdan's property (T), and show that the maximal coarse Baum-Connes conjecture fails for metric spaces possessing this property. It turns out that a residually finite group has property (T) if and only if any of its box spaces have geometric property (T). Subsequently, in \cite{WY2014}, they further characterized geometric property (T) by using the language of representations of \emph{Roe algebras} on Hilbert spaces and proved its fundamental permanence properties. Within this framework, the formulation of geometric property (T) achieves a relavant formal parallel with the classical property (T) for groups. Later, J.~Winkel generalized geometric property (T) to non-discrete spaces in \cite{Winkel2021}, and I.~Vergara provided a characterization of geometric property (T) in terms of the existence of \emph{Kazhdan projection} in \cite{Vergara2024}.

As evident from the table, there should naturally exist a corresponding notion: \emph{geometric Banach property (T)}. In recent years, the study of the $L^p$-(coarse) Baum-Connes conjecture has garnered increasing attention in noncommutative geometry, making the investigation of Roe algebra representations on Banach spaces a natural and pivotal direction. Inspired by these motivations, the primary objective of this paper is to investigate Banach space representations of Roe algebras and explore a Banach analogue of geometric property (T). Parallel to the case of group algebras, we continue to employ the language of invariant vectors and almost invariant vectors to characterize representations of Roe algebras. The description of invariant and almost invariant vectors for Roe algebra representations on Hilbert spaces has already been established in \cite{WY2014}. In this work, we adopt an analogous framework to analyze representations of Roe algebras on Banach spaces, leading to the following definition of geometric Banach property (T):

\begin{Def}[Definition \ref{def: geom T for a single space} \& \ref{def: family case of Banach Geom T}]
Let $X$ be a metric space with bounded geometry, and $\sB$ a family of Banach spaces. We say $X$ has \emph{geometric property ($T_\sB$)} if there exists $R>0$ such that, for any $\fB\in\sB$ and an isometric representation $\pi:\IC_u[X]\to\L(\fB)$, there exists $c>0$ such that
$$\sup_{\supp(V)\in\Delta_R\atop{V\text{ is a partial translation}}}\|\wt\pi(V)[\xi]-\wt\pi(\chi_{R_V})[\xi]\|\geq c,$$
for any unit vector $[\xi]\in \fB/\fB^{\pi}$, where $\fB^\pi$ is the subspace of invariant vectors, where $\wt\pi$ is the induced representation on the quotient space $\fB/\fB^{\pi}$, $R_V$ is the range of the partial translation $V$. The number $c$ is called a \emph{spectral gap}.
\end{Def}

In this work, we primarily study representations of Roe algebras on \emph{uniformly convex Banach spaces}, motivated by analogous developments in group theory. First of all, one does not have an orthogonal complement of a subspace in a Banach space. In \cite{BFGM2007}, a notion of \emph{complemented representations} for group actions on Banach spaces is introduced to compensate for the lack of orthogonal complements in general Banach spaces. In such representations, the annihilator of the invariant subspace under the dual representation precisely corresponds to a complemented direct sum of the invariant subspace of the original group representation. Crucially, the spectral gap of the group representation equals the spectral gap of the group action restricted to this annihilator subspace. It is also proved that every group representation on a uniformly convex Banach space is automatically complemented.

Naturally, we seek to adapt this framework to Banach representations of Roe algebras. However, this extension is non-trivial since metric spaces do not have group structure in general: unlike in group representations, invariant subspaces of Roe algebra representations are \emph{not} necessarily subrepresentations, as partial translations may map invariant vectors outside the subspace. To address this issue, we propose a new characterization of invariant vectors (Lemma \ref{lem: complemented representation}), showing that a vector is invariant if and only if it remains invariant under all invertible partial translations. This allows us to define and analyze complemented representations of Roe algebras in this generalized setting.

As a first taste of geometric Banach property (T), we prove the following result:

\begin{Thm}\begin{itemize}
\item[(1)](Theorem \ref{thm: TB to TB*}) Let $X$ be a separated disjoint union of a sequence of finite metric spaces which is monogenic, $\fB$ is a uniformly convex Banach space, then $X$ has geometric property ($T_{\fB}$) if and only if $X$ has geometric property ($T_{\fB^*}$).
\item[(2)](Theorem \ref{thm: coarse invariance}, coarse invariance) Let $X, Y$ be metric spaces with bounded geometry, and let $\sB$ be a uniformly convex family of Banach spaces which is closed under taking subspaces and finite direct sums. If $X$ is coarsely equivalent to $Y$, then $X$ has geometric property ($T_\sB$) if and only if $Y$ has geometric property ($T_\sB$).
\end{itemize}\end{Thm}

In addition, we will further elaborate on the concept of geometric Banach property (T) through the following interconnected perspectives.

\subsection{Kazhdan projection}

The earliest instances of the Kazhdan projection were defined in the context of a group's property (T), as seen in works such as \cite{CM1981, Valette1984, Valette1992}. In these studies, property (T) of a group could be characterized by the existence of a Kazhdan projection within the maximal group $C^*$-algebra. In the subsequence, this result has been extended in two directions. In \cite{Lafforgue2008}, the notion of Kazhdan projections is generalized to the setting of Banach property (T) by V.~Lafforgue. Moreover, Lafforgue introduce the notion of \emph{strong property (T)} by using the existence of such a projection. In \cite{DN2019}, Drutu and Nowak proved that a group has uniform property ($T_\fB$) if and only if there exists a Kazhdan projection in $C_{\fB,\max}(\Gamma)$, where $C_{\fB,\max}(\Gamma)$ denotes the completion of the group algebra $C\Gamma$ with respect to the norm
$$\|a\|_{\fB,\max}=\sup\{\|\pi(a)\|_{\fB}\mid \pi:\IC\Ga\to\L(\fB)\text{ is a representation}\},$$
for any $a\in\IC\Ga$, also see \cite{dlS2016} for the construction of Kazhdan projection and its relationship with strong Banach property (T). On the other hand, I.~Vergara proved in \cite{Vergara2024} that a bounded geometry metric space $X$ has geometric property (T) if and only if there exists a Kazhdan projection in the maximal Roe algebra associated with $X$. In this paper, we will generalize these results to the geometric version of the Banach property (T).

For a uniformly convex family of Banach spaces $\sB$, we define the norm $\|\cdot\|_{\sB,\max}$ on $\IC_u[X]$ by
$$\|a\|_{\sB,\max}=\sup\{\|\pi(a)\|\mid \pi:\IC_u[X]\to\L(\fB)\text{ is a representation},\fB\in\sB\}$$
for any $a\in\IC_u[X]$. The maximal $\sB$-Roe algebra $C_{\sB,\max}(X)$ is defined to be the completion of $\IC_u[X]$ under the norm $\|\cdot\|_{\sB,\max}$.

\begin{Thm}[Theorem \ref{thm: Kazhdan projection for Lp case}]
Let $\sB$ be a uniformly convex family of Banach spaces which is closed under taking ultraproducts. Then the following are equivalent \begin {itemize}
\item[(1)] $X$ has uniform geometric property ($T_{\sB}$);
\item[(2)] there exists an idempotent $p\in C_{\sB,\max}(X)$ such that for any representation $\pi:\IC_u[X]\to\L(\fB)$ with $\fB\in\sB$, $\pi(p)$ is the idempotent onto the invariant subspace $\fB^\pi$ along the annhilator space $\fB_{\pi}$.
\end{itemize}\end{Thm}

The Kazhdan projection in our construction arises from a functional calculus of the \emph{Laplacian operator} in Roe algebras. A key component of our proof establishes the equivalence between a Roe algebra representation having a spectral gap and the Laplacian operator possessing a spectral gap under this representation (Lemma \ref{lem: spectral gap and Markov kernel}).

\subsection{Coarse fixed point property}

In the realm of analytic group theory, there exists a profound connection between linear isometric representations of groups and their (affine) isometric group actions. This connection is particularly striking when the representation space and the action space are both Hilbert spaces. The celebrated Delorme-Guichardet theorem reveals that a countable discrete group has Kazhdan’s property (T) if and only if every isometric action of the group on a Hilbert space has a fixed point. This property is also known as the fixed point property. Delorme \cite{Del1977} and Guichardet \cite{Gui1972} independently proved the sufficiency and necessity of this theorem, respectively.

Given the correspondence between analytic group theory and coarse geometry discussed earlier, it is natural to expect an analogous fixed point property in the context of coarse geometry. In \cite{TW2022}, R. Tessera and J. Winkel introduced a coarse fixed point property for sequences of finite Cayley graphs. The bounded product of such graphs naturally carries a bornological structure, which can be viewed as a dual to topological structures. In this setting, the controlled isometric actions of the group align precisely with the coarse geometric structure (analogous to how continuous isometric actions align with topological groups). If every controlled isometric action of the bounded product group on a Hilbert space admits a fixed point, the sequence of Cayley graphs is said to have the \emph{coarse fixed point property}, denoted by coarse property $(F_\H)$. Replacing Hilbert space by a Banach space, we then have the Banach version of coarse fixed point property. In this paper, we study the relation between geometric Banach property (T) and Banach coarse fixed point property.

\begin{Thm}
Let $X=\bigsqcup_{n\in\IN}\Ga_n$ be the separated disjoint union of a sequence of finite Cayley graphs with uniformly finite generators, and $\fB$ a uniformly convex Banach space. \begin{itemize}
\item[(1)] (Theorem \ref{thm: FB to TB})
If $X$ has coarse property ($F_\fB$), then $X$ has geometric property ($T_\fB$). However, the converse does not hold.
\item[(2)] (Theorem \ref{thm: TH to FLp})
If $X$ has geometric property (T), then for any $p\in (1,2]$ and any subspace $\fB$ of any $L^p$-space $L^p(\mu)$, $X$ also has coarse property ($F_{\fB}$).
\item[(3)] (Theorem \ref{thm: lp vs l2}) 
As a corollary, if $X$ has geometric (Hilbert) property (T), then for any $p\in (1,\infty)\backslash\{2\}$, $X$ has uniform geometric property $(T_{\sL^p})$.
\end{itemize}\end{Thm}

\subsection{Residually finite groups and their Box spaces}

There exists a profound correspondence between the analytic properties of groups and the coarse geometric properties of their box spaces, analogous to the group-coarse geometry parallels discussed earlier. In \cite{WilYu2012b}, R. Willett and G. Yu proved, while first introducing geometric property (T), that a group has property (T) if and only if all its box spaces have geometric property (T). Building on this, in \cite{GQW2024}, we, in collaboration with J. Qian, introduced the concept of limit spaces, which for box spaces corresponds precisely to the residually finite groups generating them. In this work, we employ the framework of limit spaces to establish the following result:

\begin{Thm}\begin{itemize}
\item[(1)] (Proposition \ref{pro: banach property T for limit group}) Let $(\Ga_n)_{n\in\IN}$ be a sequence of finite groups and $X=\bigsqcup_{n\in\IN}$ the separated disjoint union of $(\Ga_n)$. For any $p\in(1,\infty)$ and a free ultrafilter $\omega$, if $X$ has geometric property $(T_{\sL^p})$, then $\Ga^\infty_{\omega}$ has property $(T_{\sL^p})$.
\item[(2)] (Theorem \ref{thm: box iff group in lp})
Let $\Ga$ be a finitely generated, residually finite group. For any $p\in(1,\infty)$, denote by $\sL^p$ the family of all $L^p$-spaces. Then the following are equivalent:
\begin{itemize}
\item[$\bullet$] $\Ga$ has property ($T_{\sL^p}$);
\item[$\bullet$] for any filtration $\{\Ga_n\}$, $\Box_{\{\Ga_n\}}(\Ga)$ has geometric property ($T_{\sL^p}$);
\item[$\bullet$] there exists a filtration $\{\Ga_n\}$ such that $\Box_{\{\Ga_n\}}(\Ga)$ has geometric property ($T_{\sL^p}$).
\end{itemize}
\end{itemize}
\end{Thm}

We summarize some parallel properties between residually finite groups and box spaces as follows (cf. \cite{Roe2003} for the first, \cite{CWW2013} for the second, \cite{WilYu2012b} for the third)
\begin{equation*}\begin{split}
\Gamma\mbox{ is amenable}&\iff\Box(\Gamma)\mbox{ has Yu's property A,}\\
\Gamma\mbox{ is a-T-menable}&\iff\Box(\Gamma)\mbox{ admits a fibred coarse embedding into Hilbert space,}\\
\Gamma\mbox{ has property (T)}&\iff\Box(\Gamma)\mbox{ has geometric property (T),}\\
\Gamma\mbox{ has property ($T_{\sL^p}$)}&\iff\Box(\Gamma)\mbox{ has geometric property ($T_{\sL^p}$)}.
\end{split}\end{equation*}
where $\Box(\Gamma)$ is the box space of $\Gamma$ according to any filtrations of $\Gamma$. Meanwhile, borrowing the idea of limit space, we also discuss the relationship between relative expanders with \emph{FCE-by-FCE} structure and geometric property (T). The FCE-by-FCE structure is introduced in \cite{DGWY2025}, in which we together with J.~Deng and G.~Yu prove that the coarse Novikov conjecture holds for spaces with such a structure. However, we are not able to prove the maximal coarse Baum-Connes conjecture for such spaces. Thus it is natural to ask whether such space has geometric property (T). We provide a negative answer in this paper.

\begin{Thm}[Theorem \ref{thm: FCE-by-FCE violates (T)}]
A sequence of group extensions with an FCE-by-FCE structure can never have geometric property (T).
\end{Thm}

Furthermore, we introduce a geometric Banach version of Lafforgue’s strong property (T) by the existence of the Kazhdan projection and provide a geometric characterization of a residually finite group with strong Banach property (T).

\begin{Thm}[Theorem \ref{thm: strong property (T)}]
Let $\sB$ be a uniformly convex family of Banach spaces, additionally closed under duality, conjugation, ultraproduct and $L^2$-Lebesgue-Bochner tensor product, and let $\Ga$ be a countable, discrete, residually finite group. Then the following are equivalent:\begin{itemize}
\item[(1)] $\Ga$ has strong Banach property (T) with respect to $\sB$;
\item[(2)] all box spaces of $\Ga$ have geometric strong Banach property (T) with respect to $\sB$ and $\D_\ell$;
\item[(3)] there exists a box space of $\Ga$ which has geometric strong Banach property (T) associated with $\sB$ and $\D_\ell$.
\end{itemize}\end{Thm}

\subsubsection*{Outline}

The paper is organized as follows.
In Section \ref{sec: Geometric Banach property (T)}, we introduce the concept of geometric Banach property (T) associated with a Banach space and families of Banach spaces, along with illustrative examples.
Section \ref{sec: Representations of Roe algebras on uniformly convex Banach spaces} focuses on representations of Roe algebras on uniformly convex Banach spaces and defines the notion of complemented representations for Roe algebras.
In Section \ref{sec: Kazhdan Idempotent}, we introduce the notion of Kazhdan projections and characterize geometric property ($T_{\fB}$) through their existence.
Section \ref{sec: Geometric property (T) towards infinity} explores the relationship between the geometric property ($T_{\fB}$) of a sequence of finite Cayley graphs and property ($T_{\fB}$) of their limit groups, culminating in a proof of the equivalence between (strong) property ($T_{\sL^p}$) of a residually finite group and strong geometric property ($T_{\sL^p}$) of its associated box spaces. As an application, we prove the incompatibility of FCE-by-FCE structures with geometric property (T).
Section \ref{sec: fixed point property} studies the coarse fixed point property for groups and employs this property to analyze the connection between classical geometric property (T) and its $L^p$-analogues.
In Section \ref{sec: Coarse invariance}, we prove the coarse invariance of geometric Banach property (T).
In Section \ref{sec: Open questions}, we summarize some open questions.


\section{Geometric Banach property (T)}\label{sec: Geometric Banach property (T)}

In this section, we shall recall some background on coarse geometry and introduce the notion of geometric Banach property (T) associated with a Banach space and a family of Banach spaces.

Let $(X,d)$ be a discrete \emph{extended metric space}, which means $d$ could take the value $\infty$. Throughout this paper, we shall always assume $X$ to have \emph{bounded geometry}, i.e., $\sup_{x\in X}\#B(x,R)<\infty$ for any $R>0$. The \emph{$R$-diagonal} is defined to be
$$\Delta_R=\{(x,y)\in X\times X\mid d(x,y)\leq R\}.$$
A subset $E\subseteq X\times X$ is called an \emph{entourage} if $E\subseteq\Delta_R$ for some $R>0$. The \emph{coarse structure} $\E$ of $X$ associated with this metric is the set of all entourages. An entourage is symmetric if the transposition of $E$ is equal to $E$, i.e.,
$$E^{T}=\{(x,y)\mid (y,x)\in E\}=E.$$
For two entourage $E,F\in\E$, their \emph{composition} is defined as
$$E\circ F=\{(x,y)\mid \exists z\in X\text{ such that }(x,z)\in E,(z,y)\in F\}.$$
We denoted by $E^{\circ n}=E\circ E\circ\cdots\circ E$ the $n$-times composition of $E$. It is direct to check that the composition of two entourages is still an entourage. The space $X$ is called \emph{monogenic} if there exists $E\in\E$ which generates the coarse structure of $X$, which means for any $F\in\E$, there exists $n\in\IN$ such that $F\subseteq E^{\circ n}$.

We shall use a similar convention as in \cite{WY2014}. Throughout this paper, when we say $X$ is a space, we mean it is a countable, monogenic, extended metric space with bounded geometry. Set $E_0\in\E$ to be the symmetric generating entourage. For example, $E_0$ can be taken as $\Delta_R$ for some $R\geq 0$ since $\E$ is generated by $\bigcup_{R\in\IN}\Delta_{R}$.

\begin{Def}
The \emph{algebraic uniform Roe algebra}, denoted by $\IC_u[X]$, is the set of all complex number valued $X$-by-$X$ matrices $T=(T_{xy})_{x,y\in X}$ satisfying that\begin{itemize}
\item $\sup\{|T_{xy}|\mid x,y\in X\}<\infty$;
\item the \emph{propagation} of $T$, defined by $\prop(T)=\sup\{d(x,y)\mid T_{xy}\ne 0\}$, is finite.
\end{itemize}
The \emph{support} of $T$ is defined by $\supp(T)=\{(x,y)\in X\times X\mid T_{xy}\ne 0\}$.
\end{Def}

For any $T\in\IC_u[X]$, it is direct to see that $\supp(T)\in\E$. An operator $V\in\IC_u[X]$ is called a \emph{partial translation} if $V_{xy}$ is equal to either $1$ or $0$, and for any $x\in X$, there is at most one element of the form $(x,y)$ or $(y,x)$ in $\supp(V)$. For a partial translation, there exists a \emph{local bijection} associated with this operator defined by $t_V: D_V\to R_V$ with $D_V,R_V\subseteq X$ such that the graph of $t_V$,
$$\text{graph}(t_V)=\{(x,y)\in R_V\times D_V\mid t_V(y)=x\}$$
is equal to the support of $V$, i.e., $\text{graph}(t_V)=\supp(V)$. Conversely, for a local bijection $t: D\to R$ with $\text{graph}(t)$ is an entourage, one can also define the associated partial translation
$$(V_t)_{xy}=\left\{\begin{aligned}&1,&& t(y)=x;\\&0,&&\text{otherwise.}\end{aligned}\right.$$
For convenience, such a local bijective $t$ is also called a partial translation. Define the linear map $\Phi:\IC_u[X]\to\ell^{\infty}(X)$ by
$$(\Phi(T))(x)=\sum_{y\in X}T_{xy}.$$
For a partial translation $t: D\to R$, one can easily check that $\Phi(V_t)=\chi_R$, where $\chi_R$ is the characteristic function on the range $R$ of $t$. In particular, a partial translation $V$ is \emph{full} if the range and domain of $t_V$ are both $X$.

For any $T\in\IC_u[X]$, one can always decompose $T$ into the form
$$T=\sum_{i=1}^Nf_i\cdot V_i$$
where $f_i\in\ell^{\infty}(X)$ and $V_i$ is a partial translation. Here we view $\ell^{\infty}(X)$ as a subalgebra of $\IC_u[X]$ by embedding a function $f\in\ell^{\infty}(X)$ as a diagonal operator $T_{f}$ defined by $(T_{f})_{xx}=f(x)$ and $(T_{f})_{xy}=0$ otherwise. It is direct to check that $\Phi(T)=\sum_{i=1}^Nf_i$. We define the \emph{$\ell^1$-norm} of $T\in\IC_u[X]$ to be
$$\|T\|_{\ell^1}=\inf\left\{\sum_{i=1}^N\|f_i\|_{\ell^{\infty}}\ \Big|\ T=\sum_{i=1}^Nf_iV_i\text{ with }f_i\in\ell^{\infty}(X), V_i\text{ a partial translation} \right\}.$$

\begin{Lem}
The $\ell^1$-norm of $\IC_u[X]$ is a well-defined norm.
\end{Lem}

\begin{proof}
It is direct to see the $\ell^1$-norm is positive definite and homogeneous. To show the triangle inequality, take any $T, S\in \IC_u[X]$. For any $\varepsilon>0$, we can then find a partial translation decomposition $S=\sum_{i=1}^Nf_i^SV^S_i$, $T=\sum_{j=1}^Mf_j^TV^T_j$ such that
$$\sum_{i=1}^N\|f^S_i\|_{\ell^{\infty}}\leq\|S\|_{\ell^1}+\varepsilon\quad\text{and}\quad \sum_{j=1}^M\|f^T_j\|_{\ell^{\infty}}\leq\|T\|_{\ell^1}+\varepsilon.$$
Then $S+T=\sum^N_{i=1}\sum^M_{j=1}f_i^SV^S_i+f_j^TV^T_j$ gives a partial translation decomposition for $S+T$. We then conclude that
$$\|S+T\|_{\ell^1}\leq \sum^N_{i=1}\sum^M_{j=1}\|f_i^S\|+\|f_j^T\|\leq \|S\|_{\ell^1}+\|T\|_{\ell^1}+2\varepsilon.$$
As $\varepsilon$ is arbitrarily taken, we have that $\|S+T\|_{\ell^1}\leq \|S\|_{\ell^1}+\|T\|_{\ell^1}$, this proves the triangle inequality.

For the compatibility, take any $T, S\in \IC_u[X]$ with $\|S\|_{\ell^1}=\|T\|_{\ell^1}=1$. For any $\varepsilon>0$, we still take a partial translation decomposition $S=\sum_{i=1}^Nf_i^SV^S_i$, $T=\sum_{j=1}^Mf_j^TV^T_j$ such that
$$\sum_{i=1}^N\|f^S_i\|_{\ell^{\infty}}\leq 1+\varepsilon\quad\text{and}\quad \sum_{j=1}^M\|f^T_j\|_{\ell^{\infty}}\leq 1+\varepsilon.$$
Then $ST=\sum_{i=1}^N\sum_{j=1}^Mf_i^SV^S_if_j^TV^T_j$. Write $t_i^S: D_i^S\to R_i^S$ the local bijection associated with $V_i^S$. It is direct to check that $V^S_if_j^T=(t_i^S)^*(f_j^T)V_i^S$, where
$$(t_i^S)^*(f_j^T)(x)=\left\{\begin{aligned}&f((t_i^S)^{-1}(x)),&&\text{if }x\in R_i^S\\&0,&&\text{otherwise.}\end{aligned}\right.$$
Moreover, one also has that $\|(t_i^S)^*(f_j^T)\|_{\ell^{\infty}}\leq\|f_j^T\|_{\ell^{\infty}}$. Then $ST=\sum_{i=1}^N\sum_{j=1}^M(f_i^S\cdot(t_i^S)^*(f_j^T))(V^S_iV^T_j)$. Notice that $V^S_iV^T_j$ is still a partial translation, thus this gives a partial translation decomposition of $ST$. We then conclude that
\[\begin{split}\|ST\|_{\ell^1}&\leq \sum_{i=1}^N\sum_{j=1}^M\|f_i^S\cdot(t_i^S)^*(f_j^T)\|_{\ell^{\infty}}\leq \sum_{i=1}^N\sum_{j=1}^M\|f_i^S\|_{\ell^{\infty}}\cdot\|f_j^T\|_{\ell^{\infty}}\\
&=\left(\sum_{i=1}^N\|f_i^S\|_{\ell^{\infty}}\right)\cdot\left(\sum_{j=1}^M\|f_j^T\|_{\ell^{\infty}}\right)\leq (1+\varepsilon)^2.\end{split}\]
As $\varepsilon$ is arbitrarily taken, we have that $\|ST\|_{\ell^1}\leq 1$. This finishes the proof.
\end{proof}

Let $\fB$ be a Banach space, $\L(\fB)$ the Banach algebra of all bounded linear operators on $\fB$. A representation of $\IC_u[X]$ on $\fB$ is a unital homomophism $\pi:\IC_u[X]\to\L(\fB)$. Such a represtation is called \emph{contractive} if $\|\pi(T)\|_\fB\leq\|T\|_{\ell^1}$. Throughout this paper, when we talk about a representation of $\IC_u[X]$ on a Banach space $\fB$, we shall always assume it is contractive. If $\pi$ is both contractive and unital, then it is direct to see that $\pi$ is also \emph{isometric}, i.e., for any full partial translation $V$, i.e., a partial translation determined by a bijective from $X$ to itself, $\pi(V)$ is a surjective isometric operator. From now on, when we say $\pi$ is a representation, we mean it is a contractive representation.

A vector $\xi\in \fB$ is called an \emph{invariant vector} if
$$\pi(V)\xi=\pi(\chi_{R_V})\xi$$
for any partial translation $V$, where $\chi_{R_V}=\Phi(V)$ as we discussed before. Denote by $\fB^{\pi}$ the subset of $\fB$ consisting of all invariant vectors in $\fB$. It is direct to check that the set $\fB^{\pi}$ is a closed linear subspace of $\fB$, but $\fB^{\pi}$ is not a subrepresentation in general.
Notice that $\pi$ descends to a representation $\wt\pi$ on the quotient Banach space $\fB/\fB^{\pi}$. For $\varepsilon>0$, a representation of $\IC_u[X]$ is said to admit an \emph{$(E_0,\varepsilon)$-almost invariant vector} if there exists $\xi\in\fB$ with $\|\xi\|_{\fB}=1$ such that
$$\|\pi(V)\xi-\pi(\Phi(V))\xi\|\leq\varepsilon$$
for any partial translation $V$ with $\supp(V)\subseteq E_0$. In the sequel, we shall abbreviate $\varepsilon$-almost invariant vector for $(E_0,\varepsilon)$-almost invariant vector.

\begin{Def}[Geometric Banach property (T)]\label{def: geom T for a single space}
Let $\fB$ be a Banach space, $X$ a space. A representation $\pi:\IC_u[X]\to\L(\fB)$ has a \emph{spectral gap} if there exists $c>0$ such that for any unit vector $[\xi]\in \fB/\fB^{\pi}$, there always exists a partial translation $V$ with $\supp(V)\subseteq E_0$ such that
$$\|\wt\pi(V)[\xi]-(\wt\pi(\Phi(V)))[\xi]\|\geq c,$$
where $E_0$ is the generating entourage of the coarse structure $\E$ of $X$. Simply speaking, the representation $\wt\pi$ has no $c$-almost invariant vectors for some fixed $c>0$. The supremum of such constant $c$ is always called the \emph{spectral gap} for $\pi$.

The space $X$ is said to have \emph{geometric property} ($T_\fB$) if any representation of $\IC_u[X]$ on $\fB$ has a spectral gap. The space $X$ has \emph{uniform geometric property ($T_\fB$)} if there exists $c>0$ such that any representation $\pi:\IC_u[X]\to\L(\fB)$ has a spectral gap greater than $c$.
\end{Def}

\begin{Rem}\label{rem: delete uniform if one can take direct sum}
Comparing Definition \ref{def: geom T for a single space} with the original geometric property (T) for a space, it is direct to see that geometric property ($T_\fB$) coincides with geometric property (T) introduced in \cite{WY2014} when $\fB=\H$ is an infinite-dimensional Hilbert space. In this context, there is actually no need to distinguish between uniform geometric property (T) and geometric property (T), as they are equivalent. Assume for a contradiction that $X$ has geometric property (T) but not uniform geometric property (T), then for any $n\in\IN$, there exists a representation $\pi_n: \IC_u[X]\to\L(\H)$ such that the spectral gap for $\pi_n$ is less than $\frac 1n$. Then the direct sum $\bigoplus_{n\in\IN}\pi_n$ has no spectral gap.
\end{Rem}

Parallel to the case for Hilbert space, we also have the following characteristics for spectral gap:

\begin{Lem}\label{lem: equiv characteristic}
A space $X$ has uniform geometric property ($T_{\fB}$) if and only if there exists $c>0$ such that for any representation $\pi:\IC_u[X]\to\L(\fB)$ and $[\xi]\in\fB/\fB^{\pi}$, there exists an operator $T\in\IC_u[X]$ with $\supp(T)\subseteq E_0$ such that
$$\|\wt\pi(T-\Phi(T))[\xi]\|\geq c\cdot\sup_{x,y}|T_{xy}|\cdot\|[\xi]\|.$$
\end{Lem}

\begin{proof}
The $(\Rightarrow)$ part is clear, one can directly take $T$ to be the partial translation in Definition \ref{def: geom T for a single space}. We only need to show the $(\Leftarrow)$ part. Fix a representation $\pi$ and a unit vector $[\xi]\in\fB/\fB^{\pi}$. Let $T\in\IC_u[X]$ be the operator in the statement for this $[\xi]$. Then there exists $N$ determined by $E_0$ such that $T$ can be written as $T=\sum_{i=1}^Nf_iV_i$. Without loss of generality, we can assume that $\supp(f_i)\subseteq R_{V_i}$. Since the representation is contractive, we conclude that
$$\sup\|\pi(f_i)\|\leq\sup_{x,y}|T_{xy}|.$$
We then have that
\[\begin{split}
&c\cdot\sup_{x,y}|T_{xy}|\cdot\|[\xi]\|\leq \|\wt\pi(T-\Phi(T))[\xi]\|\leq \sum_{i=1}^N\|\wt\pi(f_iV_i)[\xi]-\wt\pi(f_i)[\xi]\|\\
&\leq \sum_{i=1}^N\|\wt\pi(f_i)\|\cdot\|(\wt\pi(V_i)-\wt\pi(\chi_{R_{V_i}}))([\xi])\|\leq \sup_{x,y}|T_{xy}|\cdot\sum_{i=1}^N\|(\wt\pi(V_i)-\wt\pi(\chi_{R_{V_i}}))([\xi])\|.
\end{split}\]
As a result, there exists at least one $i$ such that
$$\|(\wt\pi(V_i)-\wt\pi(\chi_{R_{V_i}}))([\xi])\|\geq \frac cN.$$
This finishes the proof.
\end{proof}

Before further discussion, we shall first show some basic facts on geometric property ($T_{\fB}$). Recall that a \emph{separated disjoint union} of a sequence of metric spaces $(X_n)_{n\in\IN}$ is a metric space whose underlying set is the disjoint union $X=\bigsqcup_{n\in\IN}X_n$, i.e., the metric $d$ on $X$ should satisfy $d$ restricts to the original metric on each $X_n$ and $d(X_n, X_m)=\infty$ whenever $n\ne m$.

\begin{Pro}\label{pro: first look}\begin{itemize}
\item[(1)] Let $X=\bigsqcup_{n\in\IN}X_n$ be the separated disjoint union of $(X_n)_{n\in\IN}$. If the sequence $(X_n)$ is uniformly bounded, then $X$ has uniform geometric property ($T_{\fB}$) for every Banach space $\fB$.
\item[(2)] Let $X$ be a space. Then $X$ has (uniform) geometric property ($T_{C_0(X)}$) if and only if $X$ is a separated disjoint union of a uniformly finite family of metric spaces.
\end{itemize}\end{Pro}

\begin{proof}
(1) Since $(X_n)$ is uniformly bounded, we shall take the generator of the coarse structure of $X$ to be $E_0=\bigsqcup_{n\in\IN}X_n\times X_n$. By definition, we can write $\IC_u[X]=\prod_{n\in\IN}\IC_u[X_n]$ and take the generating set to be $E_0=\bigsqcup_{n\in\IN}X_n\times X_n$. Since $X_n$ is finite, we can write the Roe algebra as a matrix algebra $C_u[X_n]=M_{k_n}(\IC)$, where $k_n=\#X_n$. For each $n\in\IN$, take $P_n$ to be the averaging matrix, i.e.,
\[P_n=\begin{pmatrix}1&1&\cdots&1\\1&1&\cdots&1\\\vdots&\vdots&\ddots&\vdots\\1&1&\cdots&1\end{pmatrix}_{k_n\times k_n.}\]
Set $P=\bigoplus_{n\in\IN}P_n\in\IC_u[X]$. For any $\xi\in\fB$, it is direct to see that $P\xi$ must be an invariant vector for $\IC_u[X]$ since $VP=\Phi(V)P$ for any partial translation $V\in\IC_u[X]$.

Denote by $N=\sup\{\#X_n\mid n\in\IN\}$. We now claim that $c=\frac 1N$ satisfies Definition \ref{def: geom T for a single space}. Indeed, for any $n\in\IN$, we can always find $k_n$ surjective partial translations $u^{(n)}_1,\cdots,u^{(n)}_{k_n}$ such that
$$P_n=\sum_{i=1}^{k_n}u^{(n)}_{k_n}.$$
For example, one can find these partial translations by identifying $X_n$ with the finite group $\IZ_{k_n}$. We extend this finite sequence by define $u^{(n)}_k=0$ for any $k_n< k\leq N$. Then for each $k\in\{1,\cdots, N\}$, define $u_k=\bigoplus_{n\in\IN}u_k^{(n)}\in\IC_u[X]$ which is a partial translation. Then for any representation $\pi: \IC_u[X]\to\L(\fB)$, take $[\xi]\in\fB/\fB^{\pi}$ with $\|[\xi]\|=1$, one has that $\sum_{k=1}^N\wt\pi(u_k)[\xi]=0$ in $\fB/\fB^{\pi}$, since $\pi(P)\xi$ must be an invariant vector for $\pi$. Moreover, $\Phi(P)=\sum_{n\in\IN}k_n\chi_n\in\ell^{\infty}(X)$ is a invertible function, where $\chi_n$ is the characteristic function on $X_n$. Since $\pi$ is contractive, then the norm of the inverse function of $\Phi(P)$ satisfies that
$$\left\|\wt\pi\left(\sum_{n\in \IN}\frac1{k_n}\chi_n\right)\right\|\leq 1.$$
As a result, we conclude that $\|\wt\pi(\Phi(P))[\xi]\|\geq [\xi]$. For any unit vector $[\xi]\in\fB/\fB^{\pi}$, we then have that
\[\|\wt\pi(P)[\xi]-\wt\pi(\Phi(P))[\xi]\|=\|\wt\pi(\Phi(P))[\xi]\|\geq \|[\xi]\|.\]
By Lemma \ref{lem: equiv characteristic}, we conclude that $X$ has uniform geometric property ($T_\fB$).

(2) The ($\Leftarrow$) part follows directly from (1), it suffices to prove the ($\Rightarrow$) part. Without loss of generality, we can write $X=\bigsqcup_{n\in\IN}X_n$ such that any two points in the same $X_n$ have a finite distance. Assume for a contradiction that the sequence $\{X_n\}_{n\in\IN}$ is unbounded. Consider the ``left-regular'' representation $\lambda_n: \IC_u[X_n]\to\L(C_0(X_n))$ by
$$(Tf)(x)=\sum_{y\in X_n}T_{xy}f(y).$$
The direct sum of $\lambda_n$ defines a representation of $C_u[X]$ on $C_0(X)$. Fix a based point $x_0\in X_n$ and define $f_k\in C_0(X_n)$ for each $k\in\IN$ to be
$$f_k(x)=\frac{1}{k+d(x,x_0)}.$$
For any $R>0$, one can check that
\begin{equation}\label{eq: slow oscilation}\begin{split}|f_k(x)-f_k(y)|&=\frac{1}{k+d(x,x_0)}-\frac{1}{k+d(y,x_0)} \\&=\frac{d(x,x_0)-d(y,x_0)}{(k+d(x,x_0))(k+d(y,x_0))}\leq \frac{d(x,y)}{k^2}. \end{split}\end{equation}
If there exists $n\in\IN$ such that $X_n$ is unbounded, then there are no invariant vectors in $C_0(X_n)$. Then for any partial translation $V$ such that $\supp(V)\subseteq\Delta_R$, by \eqref{eq: slow oscilation}, one has that
$$\sup_{x\in X}|(Vf_k)(x)-f_k(x)|\leq\frac{R}{k^2}=\frac{R}{k}\cdot\|f_k\|.$$
Thus $f_k$ is a $(\Delta_R,\frac{R}{k})$-almost invariant vector for $\lambda_n$. On the other hand, if $X_n$ is bounded for each $n\in\IN$, then invariant vectors in $C_0(X_n)$ must be constant functions. Denote $D_n=\diam(X_n)$. Then $[f_k(x)]$ has norm greater than $\frac{D_n}{2k(k+D_n)}$ in $C_0(X_n)/C_0(X_n)^{\lambda}$. Take $k=D_n$, then $\|[f_k]\|\geq\frac 1{4D_n}$. Then we still have that
$$\sup_{x\in X}|(Vf_k)(x)-f_k(x)|\leq\frac{R}{D_n^2}=\frac{4R}{D_n}\cdot\|[f_k]\|.$$
For each $X_n$, there exists a $(\Delta_R,\frac{4R}{D_n})$-almost invariant vector for $\lambda_n$. As $\{D_n\}_{n\in\IN}$ is unbounded, this finishes the proof.
\end{proof}

As a corollary of the above proposition, we have the following result.

\begin{Cor}
Let $X$ be a space. The following are equivalent:\begin{itemize}
\item[(1)] $X$ is a separated disjoint union of a uniformly finite family of metric spaces;
\item[(2)] $X$ has geometric property ($T_{C_0(X)}$);
\item[(3)] $X$ has uniform geometric property ($T_{\fB}$) for every Banach space $\fB$.
\end{itemize}\end{Cor}

In particular, we shall also consider the geometric property ($T$) for a family of Banach spaces.

\begin{Def}[Geometric Banach property ($T$) for a family of Banach spaces]\label{def: family case of Banach Geom T}
Let $\sB$ be a family of Banach spaces. We say $X$ has \emph{geometric property ($T_{\sB}$)} if every representation of $\IC_u[X]$ on any $\fB\in\sB$ has a spectral gap as in Definition \ref{def: geom T for a single space}. The space $X$ has \emph{uniform geometric property ($T_{\sB}$)} if there exists $c>0$ such that all representations $\pi: \IC_u[X]\to\L(\fB)$ for any $\fB\in\sB$ have a uniform spectral gap greater than $c$.
\end{Def}

We would like to mention that the argument in Remark \ref{rem: delete uniform if one can take direct sum} holds for any family of Banach spaces which is closed under taking direct sum. For example, we denote $\sL^p$ to be the family of all $L^p$-spaces for any $p\in[1,\infty)$. Then $\sL^p$ is closed under taking $L^p$-direct sum. Thus we have the following result.

\begin{Pro}\label{pro: uniform Tlp and Tlp}
Let $X$ be a space. Then $X$ has uniform geometric property ($T_{\sL^p}$) if and only if $X$ has geometric property ($T_{L^p(\mu)}$) for any $L^p$-space $L^p(\mu)$.
\end{Pro}

\begin{proof}
The $(\Rightarrow)$ part is trivial. For the $(\Leftarrow)$ part, assume for a contradiction that $X$ does not have uniform geometric property $(T_{\sL^p})$, then for any $n\in\IN$, there exists a measure space $(Y_n,\mu_n)$ and a representation $\pi_n:\IC_u[X]\to\L(L^p(\mu_n))$ has a spectral gap smaller than $\frac 1n$. Then consider the $L^p$-direct sum of these representations $\pi:\IC_u[X]\to\L\left(\bigoplus_{p;n\in\IN}L^p(\mu_n)\right)$ clearly has no spectral gap. This leads to a contradiction.
\end{proof}

The motivation of the geometric Banach property (T) comes from the concept of \emph{Banach property (T)} for a group, which is first introduced in \cite{BFGM2007} and geometric property (T) for metric spaces first introduced in \cite{WilYu2012b}. Let $\Ga$ be a finitely generated, discrete group. Fix $S\subseteq \Ga$ a symmetric finite generating set of $\Ga$. A linear isometric representation of $\Ga$ on $\fB$ is a group homomorphism
$$\rho:\Ga\to O(\fB),$$
where $O(\fB)$ is the ``orthoginal'' group of all invertible linear isometries $\fB\to\fB$. Such a representation is said to admit an $\varepsilon$-almost invariant vector if there exists $\xi\in\fB$ with $\|\xi\|_{\fB}=1$ such that
$$\sup_{g\in S}\|\rho(g)\xi-\xi\|\leq\varepsilon.$$
A vector $\xi$ is called invariant if $\rho(g)\xi=\xi$ for any $g\in\Ga$, we still denote $\fB^{\rho}$ the space of all invariant vectors. The group $\Ga$ has \emph{property} ($T_{\fB}$) if for any representation $\rho:\Ga\to O(\fB)$,  there exists $c>0$ such that the induced representation $\wt\rho:\Ga\to O(\fB/\fB^{\rho})$ has no $c$-almost invariant vectors. The group $\Ga$ has \emph{uniform property} ($T_{\fB}$) if there exists $c>0$ such that for any representation $\rho:\Ga\to O(\fB)$, the induced representation $\wt\rho:\Ga\to O(\fB/\fB^{\rho})$ has no $c$-almost invariant vectors. The reader is referred to \cite{DN2019} for some relative discussion of Banach property (T) for a family of Banach spaces. Similarly, we also have \emph{property ($T_\sB$)} and \emph{uniform property ($T_\sB$)} for a group. From the definition above, one can view the geometric Banach property (T) as a geometric analogue of the Banach property (T) for groups.

\section{Representations of Roe algebras on uniformly convex Banach spaces}\label{sec: Representations of Roe algebras on uniformly convex Banach spaces}

In this section, we study the representations of Roe algebra on \emph{uniformly convex} Banach spaces. Parallel to Banach representations of group, we shall introduce the notion of complemented representation of Roe algebra and show that any representation of Roe algebra on a uniformly convex Banach space is complemented.

Let $\fB$ be a Banach space. The \emph{convexity modulus function} $\delta: [0,2]\to [0,1]$ of $\fB$ is defined to be
$$\delta(t)=\inf\left\{1-\frac{\|\xi+\eta\|}{2}\ \Big|\ \|\xi\|=\|\eta\|=1\text{ and }\|\xi-\eta\|\geq t\right\}$$
The Banach space $\fB$ is called \emph{uniformly convex} if the modulus function is strictly positive on $(0,2]$, i.e., $\delta(t)>0$ whenever $t\ne 0$. Recall that a uniformly convex Banach space is always reflexive, see \cite{BL2000}. Moreover, for any measured space $(X,\mu)$ and $p\in(1,\infty)$, we define the \emph{Lebesgue-Bochner space} $L^p(X,\mu,\fB)$ (or simply $L^p(\mu,\fB)$) to be the set of all $L^p$-functions from $X$ to $\fB$, equipped with the norm
$$\|\xi\|_{L^p(X,\mu,\fB)}=\left(\int_X\|\xi(x)\|^p_{\fB}d\mu(x)\right)^{\frac 1p}.$$
It is proved by M.~Day in \cite{Day1941} that if $\fB$ is uniformly convex, then the Lebesgue-Bochner space $L^p(\mu,\fB)$ is also uniformly convex.

Typical examples for our consideration of Banach Geometric property $(T)$ associated with a family of Banach spaces is \emph{a uniform convex family of Banach spaces}. A family of Banach spaces $\sB$ is \emph{uniformly convex} if the modulus of convexity of this family, defined by $\delta_{\sB}(\varepsilon)=\inf_{\fB\in\sB}\delta_{\fB}(\varepsilon)$, satisfies that $\delta_{\sB}(\varepsilon)>0$ whenever $\varepsilon>0$. When $p\in(1,\infty)$, the family $\sL^p$ is uniformly convex.

In fact, many results concerning group actions on uniformly convex Banach spaces can be extended to Roe algebras in a parallel manner. To facilitate comparison for readers, we first provide a brief review of group actions on Banach spaces. Let $\fB$ be a uniformly convex Banach space, $\Ga$ a countable discrete group. Consider an isometric group representation $\rho: \Ga\to\L(\fB)$. It induces a dual action $\rho^*:\Ga\to\L(\fB^*)$ by
$$\left(\rho^*(\gamma)f\right)(\xi)=f(\rho(\gamma^{-1})\xi)$$
where $f\in\fB^*$ and $\xi\in\fB$. It is direct to see that $\rho^*$ is also an isometric representation where $\fB^*$ is equipped with the dual norm. We shall denote $(\fB^*)^{\rho^*}$ to be all $\rho^*(\Ga)$-invariant vectors in $\fB^*$. Let $\fB_{\rho}\subseteq \fB\cong\fB^{**}$ be the annihilator of $(\fB^*)^{\rho^*}$, i.e., $\xi\in\fB_{\rho}$ if and only if $f(\xi)=0$ for all $f\in (\fB^*)^{\rho^*}$. It is proved in \cite[Proposition 2.10]{BFGM2007} that $\fB=\fB_{\rho}\oplus\fB^{\rho}$ and $\fB_{\rho}$ is isomorphic to $\fB/\fB^{\rho}$. Such a represnetation is called \emph{complemented}. Thus, to show $\Ga$ has property ($T_{\fB}$), it suffices to show that there exists $\varepsilon_0>0$ such that for any $\rho:\Ga\to\L(\fB)$, the subrepresentation $\wh\rho:\Ga\to\L(\fB_{\rho})$ has no $\varepsilon_0$-invariant vectors.

Similarly, we can also define the dual representation of the uniform Roe algebra on the dual space. For any $T\in \IC_u[X]$, the \emph{conjugate transpose} of $T$ is defined to be
$$T^*(x,y)=\overline{T(y,x)}.$$
Let $\pi:\IC_u[X]\to\L(\fB)$ be a representation. We define the \emph{dual representation} of $\pi$ to be the representation $\pi^*:\IC_u[X]\to\L(\fB)$ on the dual space $\fB^*$ given by
$$(\pi^*(T)f)(\xi)=f(\pi(T^*)\xi).$$
for any $T\in\IC_u[X]$, $\xi\in\fB$ and $f\in\fB^*$. One can also write it in the following form
\begin{equation}\label{eq: dual rep}\langle \pi^*(T)f,\xi\rangle=\langle f,\pi(T^*)\xi\rangle,\end{equation}
where the bracket $\langle,\rangle$ means the pairing between $\fB$ and $\fB^*$. Since $\fB$ is uniformly convex, by \cite[Lemma 2]{BL1962}, there exists a \emph{duality map} between the spheres of $\fB$ and $\fB^*$
$$*:S(\fB)\to S(\fB^*),\qquad \xi\mapsto f_{\xi}$$
such that $\langle f_{\xi},\xi\rangle=1$, where $f_{\xi}$ is uniquely determined by $\xi$. By \eqref{eq: dual rep}, it is direct to see that $\pi^*$ is also a contractive, isometric representation. Moreover, $\pi^{**}$ is equal to $\pi$ for any representation on uniformly convex Banach spaces.

To describe the complemented representation of Roe algebra, we need to use the following lemma, which is proved in \cite[Lemma 3.4]{Vergara2024}.

\begin{Lem}\label{lem: full partial translation}
For any monogenic metric space $X$ with bounded geometry, there exist \emph{full} partial translations $A_0,\cdots, A_n \in \IC_u[X]$ with $\supp(A_i)\subseteq E_0$ such that any partial translation $V$ with support in $E_0$ can be written as
$$V=\sum_{i=0}^n\chi_{i}A_i,$$
where $\{\chi_i\in\ell^{\infty}(X)\}$ is a set of characteristic functions of disjoint subsets of $X$ with $\Phi(V)=\bigsqcup_{i=0}^n\supp(\chi_i)$.\qed
\end{Lem}

\begin{Exa}\label{exa: explanation of Lemma 3.1}
In other words, the lemma above shows that we may not extend a partial translation to a full one, but we can split it into finitely many parts such that each part can be extended to a full partail translation.
We shall not repeat the proof of Lemma \ref{lem: full partial translation}, but it is worth providing some special examples to get some inspiration.

\paragraph{\bf Case 1.} Let $X$ be a separated disjoint union of finite Cayley graphs with uniformly finite generating sets, i.e., $X=\bigsqcup_{n\in\IN}\Ga_n$ and $\{S_n\subseteq\Ga_n\}_{n\in\IN}$ is the sequence of generating sets such that $\sup_{n\in\IN}\#S_n$. In this case, it is direct to check $X$ has bounded geometry and $X$ is monogenic with $E_0=\{(x,xs)\mid x\in\Ga_n,s\in S_n,n\in\IN\}$ as a generating set of the coarse structure. Since the generating sets $(S_n)_{n\in\IN}$ is uniformly finite, we can label the elements in each $S_n\subseteq\Ga_n$ by $S_n=\{s^{(n)}_0=e_n,s^{(n)}_1,\cdots,s^{(n)}_{k_n}\}$, where $k_n=\#S_n$. Extend this finite sequence by define $s^{(n)}_k=e_n$ for any $k_n< k\leq N$, where $N=\sup_{n\in\IN}\#S_n$. Then for each $k\in\{1,\cdots, N\}$, define
$$A_k=\bigoplus_{n\in\IN}s_k^{(n)}\in\IC_u[X]$$
which is a full partial translation.

Let $V\in\IC_u[X]$ be a partial translation such that $\supp(V)\subseteq E_0$. Denote $B_{V,0}=P_1(\Delta_X\cap\supp(V))$, where $P_1: X\times X\to X$ is the projection onto the first coordinary, i.e., $P_1(x_1,x_2)\mapsto x_1$. For any $k\in\{1,\cdots,N\}$, we denote $B_{V,k}=P_1(\supp(A_k)\cap(\supp(V)\backslash\Delta_X))$. By definition, one has that $\chi_{B_{V,k}}\cdot V=\chi_{B_{V,k}}\cdot s_k$ for each $k=\{0,1,\cdots,N\}$. Since $\bigcup_{k=1}^N\supp(s_k)=E_0$, we conclude that $\bigsqcup_{k=1}^NB_{V,k}=P_1(\supp(V))$. Denote by $\chi_{k}$ the characteristic function on $B_{V,k}$. We then have that $V=\sum_{k=0}^N\chi_{k}A_k$ and $\{B_{V,k}\}$ is a disjoint family.

\paragraph{\bf Case 2.} Consider the space of natural number with the canonical metric inherited from $\IR$. Notice that $\Delta_1$ forms a generating set of the coarse structure of $\IN$. The right-shift operator $V$ on $\ell^2(\IN)$ which maps $\delta_n$ to $\delta_{n+1}$ forms a partial translation. We define the partial translation in $\IC_u[X]$ as follow:
\[A_0=I,\quad A_1=\begin{pmatrix}
  0&  1&  &  & \\
  1&  0&  &  & \\
  &  &  0&  1& \\
  &  &  1&  0& \\
  &  &  &  &\ddots
\end{pmatrix}_{\IN\times\IN},\quad A_2=\begin{pmatrix}
  1&  &  &  &  & \\
  &  0&  1&  &  & \\
  &  1&  0&  &  & \\
  &  &  &  0&  1& \\
  &  &  &  1&  0& \\
  &  &  &  &  &\ddots 
\end{pmatrix}_{\IN\times\IN}.\]
It is straightforward to observe that the union of the supports of these three operators covers the entirety of $\Delta_1$. Although it is impossible to directly extend $V$ into a full partial translation, by defining $\chi_1$ as the characteristic function on the even-indexed set and $\chi_2$ as the characteristic function on the odd-indexed set, we immediately see that $V = \chi_1 A_1 + \chi_2 A_2$, i.e., $V$ is some combination of full partial translations. In fact, using methods analogous to those in Case 1, one can prove that any partial translation with support contained in $\Delta_1$ admits a representation satisfying Lemma \ref{lem: full partial translation}. This specific partitioning approach is precisely the method outlined in \cite[Lemma 3.4]{Vergara2024}.
\end{Exa}

\begin{Lem}\label{lem: invariant vector in group language}
Let $\pi:\IC_u[X]\to\L(\fB)$ be a representation. A vector $\xi\in\fB$ is invariant if and only if for any full partial translation $A$, one has that $\pi(A)\xi=\xi$.
\end{Lem}

\begin{proof}
It is clear to see that $(\Rightarrow)$ holds. To show $(\Leftarrow)$, let $A_0,\cdots, A_n\in \IC_u[X]$ be the full partial translations as in Lemma \ref{lem: full partial translation}. For any partial translation $V$ with $\supp(V)\subseteq E_0$, we can write $V=\sum_{n\in\IN}\chi_iA_i$ wiht $\Phi(V)=\bigsqcup_{i}\supp(\chi_i)$. Assume $\xi$ satisfies that $\pi(A)\xi=\xi$ for any full partial translation $A$. Then we have that
\[\begin{split}\pi(V)\xi&=\sum_{i=0}^n\pi(\chi_{i})\cdot\pi(A_i)\xi=\sum_{k=1}^N\pi(\chi_{i})\xi=\pi(\chi_{\bigsqcup\supp(\chi_i)})\xi=\Phi(V)\xi.
\end{split}\]
This proves that $\xi$ is invariant.
\end{proof}

For a full partial translation $A$, it is direct to see that $A^*$ is also a full partial translation such that $A^*A=AA^*=I$.

\begin{Lem}\label{lem: complemented representation}
Let $\sB$ be a uniformly convex family of Banach spaces.
\begin{itemize}
\item[(1)] For a representation $\pi: \IC_u[X]\to\L(\fB)$ with $\fB\in\sB$, let $\fB_{\pi}$ be the annihilator of $(\fB^*)^{\pi^*}$. Then $\fB\cong \fB^\pi\oplus\fB_{\pi}$. 
\item[(2)] A representation $\pi$ has a spectral gap if and only if for any $\xi\in S(\fB_{\pi})$, there exists a full partial translation $A\in\IC_u[X]$ with $\supp(A)\subseteq\Delta_R$ such that
$$\|\pi(A)\xi-\xi\|\geq c\|\xi\|.$$
\item[(3)] To show $X$ has uniform property ($T_{\sB}$), it suffices to show that for any representation $\pi:\IC_u[X]\to\L(\fB)$ with $\fB\in\sB$ and $\xi\in S(\fB_{\pi})$, there exists a full partial translation $A\in\IC_u[X]$ with $\supp(A)\subseteq\Delta_R$ such that
$$\|\pi(A)\xi-\xi\|\geq c\|\xi\|.$$
\end{itemize}\end{Lem}

\begin{proof}
(1) Let $\xi\in\fB^\pi$ with $\|\xi\|=1$. We denote $f_\xi$ the dual of $\xi$ under the $*$-map. For any full partial translation $A$, by Lemma \ref{lem: invariant vector in group language}, we conclude that
$$1=\langle \xi,f_\xi\rangle=\langle \pi(A^{*})\xi,f_\xi\rangle=\langle\xi,\pi^*(A)f_\xi\rangle.$$
Since $\pi^*$ is contractive and $*: S(\fB)\to S(\fB^*)$ is bijective, we conclude that $\pi^*(A)f_\xi=f_\xi$. Thus $f_\xi$ is $\pi^*$-invariant, i.e., $*$ restricts to a bijection from $S(\fB^{\pi})$ to $S((\fB^*)^{\pi^*})$. For any unit $\xi\in\fB^\pi$ and $\eta\in\fB_{\pi}$, if $a\xi+b\eta=0$, then we have that $f_{\xi}(a\xi+b\eta)=a=0$, thus $a=b=0$. This proves that $\fB^\pi\oplus\fB_\pi$ forms a closed subspace of $\fB$. To show it is actually dense in $\fB$, assume for a contradiction that there exists $f\in\fB^*$ with $\|f\|=1$ and $f|_{\fB^\pi\oplus\fB_\pi}=0$. Since $f|_{\fB_\pi}=0$, it means that $f$ is in the double annihilator of $(\fB^*)^{\pi^*}$. By Hahn-Banach Theorem, we conclude that $f\in (\fB^*)^{\pi^*}$, see \cite[Paragraph 2.3.6]{GTM118} for example. Since the $*$ map restricts to a bijection from $S(\fB^{\pi})$ to $S((\fB^*)^{\pi^*})$, there must exists $\lambda\in\fB^\pi$ such that $f=f_\lambda$. Thus $f|_{\fB^\pi}$ can never be $0$, which leads to a contradiction. This proves that $\fB\cong \fB^\pi\oplus\fB_{\pi}$.

(2) By the first item, for any representation $\pi:\IC_u[X]\to\L(\fB)$, there exists a direct sum decomposition $\fB\cong\fB_\pi\oplus\fB^\pi$. It induces a projection $p:\fB\to\fB_\pi$, which further induces a linear isomorphism $p:\fB/\fB^{\pi}\to\fB_\pi$ by open mapping theorem (we still denote this map by $p$ with a slight abuse of notation). Since $p:\fB/\fB^\pi\to\fB_\pi$ is an isomorphism, $p$ is both bounded and lower bounded, i.e., the quotient norm $\fB/\fB^{\pi}$ is equivalent to the norm of $\fB_\pi$. By the definition of the quotient norm, $\|p^{-1}\|\leq 1$. Thus there must exists $m\leq 1$ such that
$$\|[\xi]\|_{\fB/\fB^\pi}\in[m, 1]$$
for all $\xi\in S(\fB_\pi)$. One can then directly obtain ($\Rightarrow$) because of the equivalence of two norms. Moreover, by Lemma \ref{lem: invariant vector in group language}, any partial translation on $X$ is a uniformly finite combination (determined by bounded geometry) of full partial translations via a partition of unity. Thus it suffices to assume the supremum to be taken for all full partial translations, which implies ($\Leftarrow$).

(3) As we have explained in (2), for any representation $\pi:\IC_u[X]\to\L(\fB)$, the canonical map $p_{\pi}: \fB_\pi\to\fB/\fB^\pi$ is a contractive isomorphism. From the argument above, it suffices to prove that for a uniformly convex family of Banach spaces, the lower norm of the family $\{p_{\pi}\}$ is uniformly greater than $0$. Indeed, assume for a contradiction that there exists a sequence of representations $\{\pi_n:\IC_u[X]\to\L(\fB_n)\}_{n\in\IN}$ such that the lower norm of $p_{\pi_n}$ is smaller than $1/n$ for any $n\in\IN$. Then we can choose a unit vector $\xi_n\in(\fB_n)_{\pi_n}$ such that there exists a unit vector $\eta_n\in(\fB_n)^{\pi_n}$ such that $\|\xi_n-\eta_n\|<\frac 2n$. Let $f_{\xi_n},f_{\eta_n}$ be tha dual of $\xi_n,\eta_n$. We then conlude that $\|f_{\xi_n}-f_{\eta_n}\|\geq 1$ since $f_{\eta_n}(\xi_n)=0$ and $f_{\xi_n}(\xi_n)=1$. Notice that the duality map $\{*: S(\fB)\to S(\fB^*)\}_{\fB\in\sB}$ is equi-uniformly continuous since the uniform continuity of the dual map only depends on the convexity modulus function of $\fB$, see \cite[Proposition A.5]{BL2000} for example. This leads to a contradiction and we finish the proof.
\end{proof}

A one-sentence summary of Lemma \ref{lem: complemented representation} is that a representation $\pi: \IC_u[X]\to\L(\fB)$ is equivalent to $\pi$ has a spectral gap when restricting on $\fB_\pi$.

\begin{Thm}\label{thm: TB to TB*}
Let $X$ be a separated disjoint union of a sequence of finite metric spaces which is monogenic, $\fB$ is a uniformly convex Banach space, then $X$ has geometric property ($T_{\fB}$) if and only if $X$ has geometric property ($T_{\fB^*}$).
\end{Thm}

\begin{proof}
Assume for a contradiction that $X$ does not have geometric property ($T_\fB$). Then by Lemma \ref{lem: complemented representation}, there exists a representation $\pi: \IC_u[X]\to\L(\fB)$ such that one can choose a sequence of vectors $\{\xi_n\}_{n\in\IN}$ in $\fB_{\pi}$ such that
$$\lim_{n\to\infty}\sup_{A\in\IC_u[X]\text{: full partial translation, }\atop\supp(A)\subseteq E_0}\|\pi(A)\xi_n-\xi_n\|=0.$$
Let $\{f_{\xi_n}\}_{n\in\IN}$ be the $*$ of the sequence $\{\xi_n\}$. By the Hahn-Banach theorem and the fact that $\fB$ is uniformly convex, $\fB^*_{\pi}$ is isometrically isomorphic to $\fB^*/(\fB^*)^{\pi^*}$. Thus $\{f_{\xi_n}\}_{n\in\IN}$ defines a sequence in $S(\fB^*/(\fB^*)^{\pi^*})$. Moreover, one can also check that
$$\langle\pi(A)\xi,\pi^*(A)f_\xi\rangle=\langle\pi(A^*A)\xi,f_\xi\rangle=\langle\xi,f_\xi\rangle=1.$$
Thus the $*$ of $\pi(A)\xi$ is exactly equal to $\pi^*(A)f_\xi$. Since the $*$-map for uniformly convex Banach space is uniformly continuous, see \cite[Proposition A.4]{BL2000} for example, we conclude that
$$\lim_{n\to\infty}\sup_{A\in\IC_u[X]\text{: full partial translation, }\atop\supp(A)\subseteq E_0}\|\pi^*(A)f_{\xi_n}-f_{\xi_n}\|=0.$$
This proves that $X$ does not have geometric property ($T_{\fB^*}$).
\end{proof}

\section{Kazhdan projections}\label{sec: Kazhdan Idempotent}

In this section, we shall characterize geometric property ($T_\fB$) by the existence of a certain idempotent in some certain completion of the Roe algebra, i.e., the so-called \emph{Kazhdan projection}.

Before we get into details, we shall first need some preparation. To facilitate our discussion, we begin by establishing some common conventions and definitions that will be used throughout this section. Throughout this section, we will always assume $\fB$ to be a uniformly convex Banach space, $\sB$ to be a uniformly convex family of Banach spaces. In this section, we only assume $X$ is a separated disjoint union of a family of finite space that has uniformly bounded geometry. Assume that $X$ is monogenic and $E_0\in\E$ is the generating entourage.

\begin{Def}\label{def: maximal norm for family of representations}
Let $\F$ be a set of representations of $\IC_u[X]$ on $\fB$ (resp. $\sB$). Define the norm $\|\cdot\|_{\F}$ on $\IC_u[X]$ by
$$\|a\|_{\F}=\sup\{\|\pi(a)\|\mid \pi\in\F\}$$
for any $a\in\IC_u[X]$. If $\F$ is the set of all representations on $\fB$ (resp. $\sB$), then we shall denote the norm by $\|\cdot\|_{\fB,\max}$ (resp. $\|\cdot\|_{\sB,\max}$).

The algebra $C_{\F}(X)$ is defined to be the completion of $\IC_u[X]$ under the norm $\|\cdot\|_\F$. If $\F$ is the set of all representations on $\fB$, then we shall denote this Banach algebra by $C_{\fB,\max}(X)$ (resp. $C_{\sB,\max}(X)$).
\end{Def}

If $\sB=\sL^p$, then the algebra $C_{\sL^p,\max}(X)$, which is also denoted by $C^p_{\max}(X)$, is called the \emph{maximal uniform $L^p$-Roe algebra}. When $\sB$ is the family of all Hilbert spaces, then $C_{\fB,\max}(X)$ is exactly the \emph{maximal Roe $C^*$-algebra} $C^*_{\max}(X)$. Analogue to the uniform Roe algebra and its maximal version, $C^p_{\max}(X)$ is a natural generalization, which is relevant to the $L^p$-version of the maximal coarse Baum-Connes conjecture. Fix a family of representation $\F$. For any representation $\pi\in\F$, the representation $\pi$ will give $\IC_u[X]$ a completion, which is denoted by $C_{\pi}(X)$. By the universal property of the norm  $\|\cdot\|_{\F}$, there exists a canonical quotient homomorphism
$$Q_{\pi}: C_{\F,\max}(X)\to C_{\pi}(X).$$

For our convenience, we shall denote $A_1,\cdots,A_n\in\IC_u[X]$ to be the full partial translation as in \cite[Lemma 3.4]{Vergara2024}, which means that for any partial translation $V$ with $\supp(V)\subseteq E_0$, there exists disjoint subsets $B_0,\cdots, B_n\in\ell^{\infty}(X)$ such that
$$V=\sum_{i=1}^n\chi_{B_i}A_i.$$

We define the \emph{Laplacian} associated with the decompostion $\{A_1,\cdots,A_n\}$ to be
$$\Delta=1-\frac 1n\cdot\sum_{i=1}^nA_i$$
Define $A\in\IC_u[X]$ by
$$A=1-\frac{\Delta}2=\frac{1}{n}\sum_{i=1}^n\frac{1+A_i}2.$$
Notice that for any representation $\pi:\IC_u[X]\to\L(\fB)$, the subspaces $\fB_{\pi}$ and $\fB^{\pi}$ are both invariant under the action of $\pi(A_i)$ by Lemma \ref{lem: complemented representation}. Set $p_{\pi}$ to be the idempotent onto $\fB^{\pi}$ along $\fB_\pi$. We then conclude that $\pi(A)$ commutes with $p_\pi$. On the other hand, for any $\xi\in\fB^{\pi}$, one has that $\pi(A_i)\xi=\xi$, thus $\pi(A)\xi=\xi$. We then conclude that $p_{\pi}=\pi(A)p_{\pi}=p_\pi\pi(A)$, the last equality is because of the fact that $p_\pi$ commutes with $\pi(A)$.

\begin{Lem}\label{lem: spectral gap and Markov kernel}
Let $\fB$ be a uniformly convex Banach space, $\pi:\IC_u[X]\to\L(\fB)$ a representation. Then the following are equivalent:\begin{itemize}
\item[(1)] $\pi$ has a spectral gap $c>0$;
\item[(2)] $\|\pi(A)|_{\fB_\pi}\|=\sup_{\xi\in\fB_{\pi}}\frac{\|\pi(A)\xi\|}{\|\xi\|}=\lambda<1;$
\item[(3)] the limit of $\pi(A^k)$ converges to $p_\pi$ with a summable rate, i.e., the series $\sum_{k=1}^\infty\|\pi(A^k)-p_\pi\|$ converges to some $S>0$.
\end{itemize}
Moreover, the parameters $c$, $\lambda$, and $S$ can be mutually determined from one another.
\end{Lem}

\begin{proof}
$(1)\Rightarrow (2)$ By assumption that $\pi$ has a spectral gap, by Lemma \ref{lem: complemented representation}, there exists a constant $c>0$ such that any unit vector $\xi\in\fB_{\pi}$, one of the partial translation stated above $A_i\in\IC_u[X]$ should satisfy that
$$\|\pi(A_i)\xi-\xi\|\geq c.$$
Since $A_i$ is a full partial translation and $\pi$ is isometric, $\|\pi(A_i)\xi\|=1$. Since $\fB$ is uniformly convex, there exists $\delta\in(0,1)$ only depends on the spectral gap $c$ and the convexity modulus function, such that
$$\left\|\frac{\pi(A_i)\xi+\xi}2\right\|\leq\delta.$$
Notice that
\[\begin{split}\|\pi(A)\xi\|&\leq \frac 1n\left\|\frac{\pi(A_i)\xi+\xi}2\right\|+\frac 1n\sum_{j\ne i}\left\|\frac{\pi(A_j)\xi+\xi}2\right\|\\
&\leq \frac{\delta}n+\frac{n-1}n<1.\end{split}\]
Since $\xi$ is arbitrarily taken, this inequality holds for all unit vectors in $\fB_{\pi}$. 

$(2)\Rightarrow (3)$ As we discussed above, we have that
\[\|\pi(A)-p_\pi\|=\|\pi(A)-\pi(A)p_\pi\|=\|\pi(A)(1-p_\pi)\|=\|\pi(A)|_{\fB_\pi}\|\leq\lambda<1.\]
For any integer $k\geq 1$, we claim that
$$(\pi(A)-p_\pi)^k=\pi(A^k)-p_\pi.$$
Indeed, this equation holds for $k=1$ and we assume it holds for $k\leq N-1$. When $k=N$, we have that
\[\begin{split}(\pi(A)-p_\pi)^k&=(\pi(A)-p_\pi)(\pi(A^{k-1})-p_\pi)\\
&=\pi(A^k)-p_\pi\pi(A^{k-1})-\pi(A)p_\pi+p_\pi\\
&=\pi(A^k)-2p_\pi+p_\pi=\pi(A^k)-p_\pi.\end{split}\]
Thus, we conclude that
$$\|\pi(A^k)-p_\pi\|=\|(\pi(A)-p_\pi)^k\|\leq\|\pi(A)-p_\pi\|^k\leq\lambda^k\to 0,\quad\text{as }k\to\infty,$$
and the convergence is in a summable rate. 

$(3)\Rightarrow(1)$ Fix $\xi\in\fB_{\pi}$. Denote $a_k=\|\pi(A^k)-p_\pi\|=\|\pi(A)|_{\fB_{\pi}}\|$ and $c_{\xi}=\sup_{i\in\{1,\cdots,n\}}\|\pi(A_i)\xi-\xi\|$. Write $S=\sum_{k=1}^\infty a_k$. Then
$$\|\pi(A)\xi-\xi\|\leq \frac 1n\sum_{j=1}^n\left\|\frac{\pi(A_j)\xi-\xi}2\right\|\leq {c_\xi}.$$
Then for any $k\in\IN$, one has that
\[\begin{split}\|\pi(A^k)\xi-\xi\|&\leq \sum_{i=1}^k\|\pi(A^i)\xi-\pi(A^{i-1})\xi\|\\
&\leq \sum_{i=1}^k\|\pi(A^{i-1})|_{\fB_\pi}\|\cdot\|\pi(A)\xi-\xi\|\\
&\leq c_\xi\cdot\left(1+\sum_{i=1}^{k-1}a_i\right)\leq c_\xi(1+S).\end{split}\]
As $k$ tends to infinity, $\|\pi(A^k)\xi-\xi\|$ tends to $1$. Thus, $1\leq c_\xi(1+S)$. As a result, $c_\xi\geq\frac 1{1+S}$ for any $\xi\in\fB_\pi$.
This finishes the proof.
\end{proof}

Lemma \ref{lem: spectral gap and Markov kernel} shows that the Laplacian $\Delta$ acts as a $0$ function on $\fB^{\pi}$, and $\|\Delta|_{\fB_\pi}\|$ is strictly greater than $0$. This somehow shows that $\pi(\Delta)$ indeed has a spectral gap. With the aforementioned preparations in place, we are now poised to introduce the central concept of this section, namely, the \emph{Kazhdan projection}.

\begin{Def}\label{def: Kazhdan projection}
An element $p\in C_\F(X)$ is called a \emph{Kazhdan projection} if\begin{itemize}
\item[(1)] $Q_\pi(p)=p_\pi$ for any representation $\pi\in\F$, where $Q_\pi: C_\F(X)\to C_\pi(X)$;
\item[(2)] the sequence $\{A^k\}_{k\in\IN}$ converges to $p$ in $C_\F(X)$ with a summable rate.
\end{itemize}\end{Def}

The following remark will provide some initial insights into the Kazhdan projection.

\begin{Rem}
The Kazhdan projection $p\in C_{\F}(X)$ is indeed an idempotent. One can check that
$$\|p^2-p\|_{\F}=\lim_{k\to\infty}\|A^{2k}-A^k\|$$
Since $\{A^k\}_{k\in\IN}$ is Cauchy, the limit above tends to $0$, thus $p$ is indeed an idempotent. Moreover, $p$ commutes with all $A_i$. Indeed, there is a canonical inclusion homomorphism:
$$\iota: C_{\F}(X)\to\bigoplus_{\pi\in\F}C_\pi(X),$$
where the norm on the right side is given by $\|(a_{\pi})\|_{\pi\in\F}=\sup_{\pi}\|a_{\pi}\|_{C_\pi(X)}$ and the inclusion map $\iota$ is given by $a\mapsto (Q_\pi(a))_{\pi\in\F}$. It is direct to see that $\iota$ is an isometry on $\IC_u[X]$, thus giving an injection. Notice that $\iota(p)=(p_\pi)$ and $p_\pi$ commutes with all $\pi(A_i)=Q_\pi(A_i)$. Thus $\iota(A_i)$ commutes with $\iota(p)$ on the right side. Since $\iota$ is an injection, we conclude that $p$ commutes with all $A_i$.
\end{Rem}

The following is the main theorem of this section.

\begin{Thm}\label{thm: property T via Kazhdan projection}
Let $X$ be a separated disjoint union of finite spaces with bounded geometry such that $X$ is monogenic. Then the space $X$ has uniform geometric property ($T_{\sB}$) if and only if the algebra $C_{\sB,\max}(X)$ has a Kazhdan projection.
\end{Thm}

\begin{proof}
For the $(\Leftarrow)$ part, assume that $C_{\sB,\max}(X)$ admits a Kazhdan projection. Thus $\lim_{k\to\infty}A^k=p$ and $\sum_{k=1}^\infty\|A^k-p\|_{\sB,\max}$ converges to $S$. From the definition of $\|\cdot\|_{\sB,\max}$, for any representation $\pi:\IC_u[X]\to\L(\fB)$ with $\fB\in\sB$, one has that $\sum_{i=1}^n\|\pi(A^k)-p_\pi\|_\fB\leq S$. Thus by Lemma \ref{lem: spectral gap and Markov kernel}, $\pi$ has a spectral gap greater than $\frac 1{S+1}$, i.e., $X$ has uniform geometric property ($T_{\sB}$).

For the $(\Rightarrow)$ part, assume that $X$ has uniform geometric property ($T_{\sB}$). For any representation $\pi$, for any $m\leq n$, one can check that
\[\|\pi(A^n)-\pi(A^m)\|=\|\pi(A^m)(\pi(A^{n-m})-1)\|\leq 2\|\pi(A^m)|_{\fB_{\pi}}\|\]
By Lemma \ref{lem: spectral gap and Markov kernel} and the fact that $\sB$ is uniformly convex, for any representation and $\fB\in\sB$,
$$\|\pi(A^m)|_{\fB_\pi}\|\leq\left(\frac{\delta+n-1}n\right)^m\to 0\quad \text{ as }m \to\infty.$$
Thus $\{A^k\}$ forms a Cauchy sequence in $C_{\sB,\max}(X)$ and $A^k$ converges to $p$ within a summable rate. By Lemma \ref{lem: spectral gap and Markov kernel}, $Q_{\pi}(p)=p_\pi$ for any representation $\pi$. This finishes the proof.
\end{proof}

Comparing Definition \ref{def: Kazhdan projection} with the traditional version of the Kazhdan projection for representations on Hilbert space, such as that found in \cite{Valette1984, Vergara2024}, it becomes evident that condition (1) is relatively more natural. From the perspective of Lemma \ref{lem: spectral gap and Markov kernel}, condition (2) also appears natural, although it carries a more technical flavor than its traditional counterpart. However, if we assume that the family of Banach spaces $\sB$ is sufficiently well-behaved, then condition (2) can effectively be rendered implicit.

We say a family of $\sB$ is \emph{closed under taking ultraproduct} if for any sequence $\{\fB_n\}_{n\in\IN}$ in $\sB$ and any non-principal ultrafilter $\omega\in\beta\IN$, where $\beta\IN$ is the set of all ultrafilters on $\IN$, the ultraproduct $\fB_{\omega}=\prod_{\omega}\fB_n$ is still an element in $\sB$. The reader is referred to \cite[Section 14.1]{Goldbring2022} for some discussion on the ultraproduct of a sequence of Banach space. For the convenience of the reader, we shall include a short explanation here. For a fixed ultrafilter $\omega$, a sequences $(\xi_n)\in\prod_{n\in\IN}\fB_n$ is said to be $C_{0,\omega}$ if
$$\lim_{n\to\omega}\|\xi_n\|_{\fB_n}=0.$$
Notice that the space
$$\prod_{C_{0,\omega}}\fB_n=\left\{(\xi_n)\in\prod_{n\in\IN}\fB_n\ \Big|\ (\xi_n) \text{ is }C_{0,\omega}\right\}$$
forms a closed subspace of $\prod_{n\in\IN}\fB_n$.
The ultraproduct of $\{\fB_n\}_{n\in\IN}$, denote by $\fB_{\omega}$, is defined to be the quotient Banach space
$$\fB_{\omega}=\frac{\prod_{n\in\IN}\fB_n}{\prod_{C_{0,\omega}}\fB_n}$$
and the norm on $\fB_{\omega}$ is defined by
$$\|[\xi_n]\|=\lim_{n\to\omega}\|\xi_n\|.$$
It is direct to see that $\mathscr H$, the family of all Hilbert spaces, is closed under taking ultraproduct. It is also known that the family $\sL^p$ is closed under taking ultraproduct for any $p\in[1,\infty)$, see \cite{DK1970}.

\begin{Lem}
Let $\sB$ be a uniformly convex family of Banach spaces. Denote by $\overline\sB$ to be the smallest family of all Banach spaces containing $\sB$ which is closed under taking ultraproduct. Then $\overline\sB$ is still a uniformly convex family of Banach spaces.
\end{Lem}

\begin{proof}
For any ultrafilter $\omega\in\beta\IN$ and a sequence $\{\fB_n\}_{n\in\IN}\subseteq\sB$, from the definition, one can check that $\delta_{\fB_{\omega}}(\varepsilon)\geq\delta_{\sB}(\varepsilon)$ since $\delta_{\fB_n}(\varepsilon)\geq\delta_{\sB}(\varepsilon)$ for any $n\in\IN$.
\end{proof}

\begin{Thm}\label{thm: Kazhdan projection for Lp case}
Let $\sB$ be a uniformly convex family of Banach spaces which is closed under taking ultraproduct, $X$ a separated disjoint union of finite spaces with bounded geometry such that $X$ is monogenic. Then the following are equivalent\begin{itemize}
\item[(1)] $X$ has uniform geometric property ($T_{\sB}$);
\item[(2)] there exists a idempotent $p\in C_{\sB,\max}(X)$ such that $Q_\pi(p)=p_\pi$ for any representation $\pi:\IC_u[X]\to\L(\fB)$ with $\fB\in\sB$.
\end{itemize}\end{Thm}

\begin{proof}
(1) $\Rightarrow$ (2) is implied by Theorem \ref{thm: property T via Kazhdan projection}, we shall only prove (2) $\Rightarrow$ (1). Assume for a contradiction that $X$ does not have uniform geometric property ($T_\sB$), i.e., for any $n\in\IN$, there exists a representation $\pi_n:\IC_u[X]\to\L(\fB_n)$ and a unit vector $\xi_n\in(\fB_n)_{\pi}$ such that for any full partial translation $V\in\IC_u[X]$ with $\supp(V)\subseteq E_0$, one has that
$$\|\pi(V)\xi_n-\xi_n\|\leq \frac 1n.$$
Define the representation $\pi_\omega: \IC_u[X]\to\L(\fB_\omega)$ to be
$$\|\pi_{\omega}(T)[\eta_n]\|=[\pi_n(T)\eta_n]$$
for any $[\eta_n]\in\fB_\omega$. It is direct to see the representation $\pi_\omega$ is an isometric representation since all representations $\pi_n$ are isometric. Moreover, by the universal property of the maximal norm, these representations extend canonically to $\pi_n: C_{\sB,\max}(X)\to\L(\fB_n)$ and $\pi_\omega: C_{\sB,\max}(X)\to\L(\fB_\omega)$. Consider $[\xi_n]\in\fB_\omega$, for any full partial translation with $\supp(V)\subseteq E_0$, we have that
$$\|\pi_\omega(V)[\xi_n]-[\xi_n]\|=\lim_{n\to\omega}\|\pi_n(V)\xi_n-\xi_n\|=0.$$
Thus $[\xi_n]$ is an invariant vector for $\pi_\omega$. However, take the Kazhdan projection $p\in C_{\sB,\max}(X)$, one has that
$$1=\|[\xi_n]\|=\|Q_{\pi_\omega}(p)[\xi_n]\|=\|\pi_\omega(p)[\xi_n]\|=\lim_{n\to\omega}\|\pi_n(p)\xi_n\|=0,$$
the last equation follows from that $\xi_n\in\fB_\pi$. This leads to a contradiction.
\end{proof}

As a corollary of Proposition \ref{pro: uniform Tlp and Tlp} and Theorem \ref{thm: Kazhdan projection for Lp case}, we have the following result. One can compare it with \cite[Theorem 1.1]{Vergara2024}.

\begin{Cor}
Let $X$ be a separated disjoint union of finite spaces with bounded geometry such that $X$ is monogenic. Then for any $p\in(1,\infty)$, the following are equivalent\begin{itemize}
\item[(1)] $X$ has geometric property ($T_{\sL^p}$);
\item[(2)] there exists a idempotent $p\in C^p_{\max}(X)$ such that $Q_\pi(p)=p_\pi$ for any representation $\pi:\IC_u[X]\to\L(L^p(\mu))$.\qed
\end{itemize}
\end{Cor}

One can check that the idempotent $P$ in Proposition \ref{pro: first look} is exactly a Kazhdan projection.

\section{Geometric property ($T_\fB$) at infinity}\label{sec: Geometric property (T) towards infinity}

In \cite{WilYu2012b}, one of the motivations of the geometric property (T) is to characterize Kazhdan's property (T) for a residually finite group by using the coarse geometric behavior of its box space. It is proved that a residually finite group $\Ga$ has property (T) if and only if any of its box spaces $\Box(\Ga)$ has geometric property (T). In this section, we shall discuss a parallel result on the geometric property ($T_{\fB}$). Moreover, we actually aim to prove a stronger result in the framework of \emph{limit groups}. In \cite{GQW2024}, the authors of this paper, along with J.~Qian, employed the ultraproduct construction to reprove the limit space theory initiated by R. Willett and J.~{\v S}pakula in \cite{SpaWil2017}. In this section, we will apply this set of techniques to a metric space consisting of a sequence of \emph{finite groups}.

\subsection{Primilinaries on limit space theory}

For a sequence of uniformly finitely generated finite groups $(\Ga_n)_{n\in\IN}$, we shall choose a sequence of uniformly finite symmetric generating set $S_n\subseteq\Ga_n$ for each $n\in\IN$. Then $E_0=\{(x,xs)\mid x\in\Ga_n,s\in S_n,n\in\IN\}$ forms a generating set of the coarse structure. Let $l_n$ be the length function on $\Ga_n$ associated with the generating set $S_n$, and we define $d_n(g,h)=l_n(g^{-1}h)$ for any $g,h\in\Ga_n$. Notice that $d_n$ is \emph{left} invariant. Throughout this section, we shall always assume that $X=\bigsqcup_{n\in\IN}\Ga_n$ is the separated disjoint union of $\{\Ga_n\}$. Denote $E=\{e_n\mid{n\in\IN}\}$ the set of all unit elements $e_n\in\Ga_n$.

Fix $\omega\in\partial_\beta\IN$ to be a free ultrafilter. A sequence $\{x_i\}_{i\in\IN}\subseteq X$ is said to be \emph{afar} if $\{x_i\}\cap\left(\bigsqcup_{n=1}^N\Ga_n\right)$ is finite. Denote by $X^{\omega}$ the set-theoretic ultraproduct of $X$, i.e., the set of all sequences in $X$ modulating the equivalent relationship
$$(x_i)\sim_\omega (y_i)\iff\{i\in\IN\mid x_i=y_i\}\in\omega.$$
The equivalent class determined by $(x_i)$ is denoted by $[x_i]$. Denote by $\IR^{\omega}$ the set-theoretic ultraproduct of $\IR$, which is a model of \emph{hyperreal numbers}. Then $X^{\omega}$ becomes a hyperreal-valued metric space, with metric function $d^{\omega}$ defined by
$$d^{\omega}([x_i],[y_i])=[d(x_i,y_i)]\in\IR^{\omega}\cup\{\infty\}.$$
Fix an afar element $x\in X^{\omega}$, where $x$ is the equivalent class of $(x_i)$, denoted by
$$\Ga^{\infty}_{\omega,x}=\{y\in X^{\omega}\mid d^{\omega}(x,y)\in\IR\subseteq\IR^{\omega}\cup\{\infty\}\},$$
which is called the \emph{limit space} of $X$ associated with $x$. Since $X$ is strongly discrete with bounded geometry, i.e., $d(g,h)\in\IN$ for any $g, h\in\Ga_n$, and $n\in\IN$, it is proved in \cite[Proposition 2.6]{GQW2024} that $\Ga^{\infty}_{\omega,x}$ is still a metric space with bounded geometry with the metric defined to be the restriction of $d^{\omega}$ on $\Ga^{\infty}_{\omega,x}$. The reader is referred to \cite[Section 2]{GQW2024} for more details.

An afar element $x\in X^{\omega}$ is called \emph{unit} if it is an equivalent class of an afar sequence in the unit set $E$. For a unit afar element $x$, the limit space $\Ga^{\infty}_{\omega,x}$ has a canonical group structure. Write $x=[a_i]\in E^{\omega}$, denote $n(i)\in\IN$ to be the index such that $a_i=e_{n(i)}\in\Ga_{n(i)}$. If $y=[y_i]\in\Ga^{\infty}_{\omega,x}$, then it means that $D_y=\{i\in\IN\mid y_i\in\Ga_{n(i)}\}\in\omega$ by definition. Then for any $y,z\in\Ga^{\infty}_{\omega,x}$, the multiplication of these two elements is defined to be $yz=[(yz)_i]$, where
\[(yz)_i=\left\{\begin{aligned}&y_iz_i,&&i\in D_y\cap D_z;\\&e_i,&&\text{otherwise.}\end{aligned}\right.\]
The inverse of $y$ is defined to be $y^{-1}=[y^{-1}_i]$. It is direct to see that $\Ga^{\infty}_{\omega,x}$ forms a group under the multiplication and inverse defined above with the unit element given by $x$. Thus $\Ga^{\infty}_{\omega,x}$ is also called a \emph{limit group} of $X=\bigsqcup_{n\in\IN}\Ga_n$. Actually, for any afar element $y=[y_i]\in\Ga^{\omega}$, let $x=[e_i]$ to be such that $e_i,x_i\in\Ga_{n(i)}$ for all $i\in\IN$. Then limit space $\Ga^{\infty}_{\omega,y}$ associated with $y$ is isometric to the limit group $\Ga^{\infty}_{\omega,x}$, the isometry is given by
$$\Ga^{\infty}_{\omega,y}\to \Ga^{\infty}_{\omega,x},\quad z\mapsto y^{-1}z$$
where $y^{-1}z$ is defined in the forms of multiplication and inverse as above. Thus, for such space $X$, it suffices to only consider the limit groups instead of all limit spaces.

\begin{Lem}\label{lem: finitely generated limit group}
For any unit afar element $x$, the limit group $\Ga^{\infty}_{\omega,x}$ is finitely generated.
\end{Lem}

\begin{proof}
Say $x=[a_i]$ and $n(i)\in\IN$ to be such that $a_i=e_{n(i)}\in\Ga_{n(i)}$. For each $n\in\IN$, denote $S_n$ the generating set of $\Ga_n$.
Take $S^{\infty}_{\omega,x}=\{[b_i]\in G^{\infty}_{\omega,x}\mid b_i\in S_{n(i)}\}$, thus $S^{\infty}_{\omega,x}$ forms a generating set of $G^{\infty}_\omega$. For each $n\in\IN$, write $S_n=\{s_{n,1},\cdots, s_{n,k_n}\}$, where $k_n=\#S_n$. Since the sequence $(S_n)_{n\in\IN}$ is uniformly finite, set $N=\sup_{n\in\IN}k_n$. Then for a fixed sequence $(b_i)_{i\in\IN}$ satisfying that $b_i\in S_{n(i)}$ and $j\in\{1,\cdots,N\}$, define the set $D_j=\{i\in\IN\mid b_i=s_{n(i),j}\}$. Then $\{D_j\}_{j=1}^{N}$ forms a disjoint cover of $\IN$, by definition of ultrafilter, there will be an unique $j_0$ such that $D_{j_0}\in\omega$. Thus $(b_i)_{i\in\IN}$ is equivalent to the sequence
$$s^{(j_0)}_i=\left\{\begin{aligned}&e_{n(i)},&& {j_0>k_{n(i)}};\\&s_{n(i),j_0},&&j_0\leq k_{n(i)},\end{aligned}\right.$$
since it is equal to $(s^{(j_0)}_i)_{i\in\IN}$ on $D_{j_0}$. Thus $S^{\infty}_{\omega,x}$ is actually a finite set with at most $N$ elements (we should mention that $s^{(j)}=[s^{(j)}_i]$ could be equal to the unit element $x$ for large $j$). This finishes the proof.
\end{proof}

We denote $e\in X^{\omega}$ the element determined by the sequence $\{i\mapsto e_i\}_{i\in\IN}$, called the \emph{fundamental} afar element. For fixed ultrafilter $\omega\in\partial_{\beta}\IN$, the limit group associated with the fundamental afar element and $\omega$ is much easier to see. Denote by $\prod^b_{i\in\IN}\Ga_i$ be the set of all sequences $(\gamma_i)\in\prod_{i\in\IN}\Ga_i$ such that $\sup_{i\in\IN}l_i(\gamma_i)<\infty$. Notice that $\prod^b_{i\in\IN}\Ga_i$ is a group under pointwise multiplication. We define
$$\N_{0,\omega}={\left\{(\gamma_i)\in \prod_{i\in\IN}\Ga_i\mid \lim_{i\to\omega}l_i(\gamma_i)=0\right\}}$$
It is direct to see that $\N_{0,\omega}$ is a normal subgroup of $\prod_{i\in\IN}^b\Ga_i$. Then the \emph{limit group} of $(\Ga_i)_{i\in\IN}$ associated with $\omega$ is defined to be
$$\Ga^{\infty}_{\omega}=\frac{\prod^b_{i\in\IN}\Ga_i}{\N_{0,\omega}},$$
we shall abbreviate $\Ga^{\infty}_{\omega,e}$ as $\Ga^{\infty}_{\omega}$ for simplicity if the base point is chosen as the fundamental afar element.

\begin{Lem}\label{lem: limit group with fundamental base}
For any unit afar element $x$ and $\omega\in\partial_{\beta}\IN$, there exists $\mu\in\partial_\beta\IN$ such that $\Ga^{\infty}_{\omega,x}\cong\Ga^{\infty}_{\mu}$.
\end{Lem}

\begin{proof}
Say $x=[(e_{n(i)})_{i\in\IN}]\in\Ga^{\omega}$, where $n: i\mapsto n(i)$ is a unbounded map determined by the unit afar element $x$ as above. Then $n:\IN\to\IN$ extends to a continuous map $\wt n:\beta\IN\to\beta\IN$, we shall denote $\mu=\wt n(\omega)\in\partial_\beta\IN$, i.e., $A\in\mu$ if and only if there exists $B\in\omega$ such that $n(B)\subseteq A$.

Define $\varphi:\Ga^{\infty}_{\mu}\to\Ga^\infty_{\omega,x}$ to be
$$[\gamma_i]\mapsto [\gamma_{n(i)}].$$
If $(\gamma_i)\sim_{\mu}(\gamma'_i)$, then $\{i\in\IN\mid \gamma_i=\gamma'_i\}\in\mu$. By definition, there exist $B\in\omega$ such that $\{n(i)\mid i\in B\}\subseteq\{i\in\IN\mid \gamma_i=\gamma'_i\}$. Thus, for any $i\in B$, one has that $\gamma_{n(i)}=\gamma'_{n(i)}$, which means that $(\gamma_{n(i)})\sim_{\omega}(\gamma'_{n(i)})$. This proves the map $\varphi$ is well-defined. It is direct to see this map is a group homomorphism. To see it is injective, if $\varphi([\gamma_i])=x$, then there exists $B\in\omega$ such that $\gamma_{n(i)}=e_{n(i)}$ for all $i\in B$. Thus, $\gamma_i=e_i$ for all $i\in n(B)\in\mu$, which means that $[\gamma_i]=e$. To see $\varphi$ is surjective, for any $y\in\Ga^{\infty}_{\omega,x}$, there exists $R>0$ such that $d(x,y)\leq R$. Notice that there are only finitely many elements in $B(x,R)$. With a similar proof as Lemma \ref{lem: finitely generated limit group}, one can see that $y$ can be represented by a sequence $(y_{(n(i))})$ such that $y_{n(i)}$ and $y_{n(j)}$ whenever $n(i)=n(j)$ for all $i,j\in B\in\omega$. For each $k=n(i)\in n(B)\IN$, just define $z_k=y_{n(i)}$, otherwise take $z_i=e_i$. Then $\varphi([z_i])=y$. This finishes the proof.
\end{proof}

By Lemma \ref{lem: limit group with fundamental base}, it suffices to consider the family of $\{\Ga^{\infty}_{\omega}\}_{\omega\in\partial_\beta\IN}$, which is much more convenient to discuss.

For a uniformly finite generated sequence of finite group extension $(1\to N_n\to \Ga_n\to Q_n\to 1)_{n\in\IN}$, we choose $(S_n\subseteq \Ga_n)_{n\in\IN}$ to be a uniformly finite generating set. Then $N_n$ is equipped with the subspace metric and $Q_n$ is equipped with the quotient metric. Then the length function on $N_n$ and $Q_n$ are given by $l_{N_n}(a)=d_{N_n}(a,e)$ and $l_{Q_n}(b)=d_{Q_n}(b,e)$, respectively.

\begin{Lem}\label{lem: exact sequence of limit groups}
Fix $\omega\in\partial_\beta\IN$. If $(1\to N_n\to \Ga_n\to Q_n\to 1)_{n\in\IN}$ be a sequence of uniformly finite generated finite group extensions, then there is a short exact sequence on their limit group
$$1\to N^{\infty}_{\omega}\to\Ga^{\infty}_{\omega}\to Q^{\infty}_{\omega}\to 1.$$
\end{Lem}

\begin{proof}
Set $\iota_n: N_n\to \Ga_n$ and $\pi_n: \Ga_n\to Q_n$. We then define $\iota^{\infty}_{\omega}:N^{\infty}_{\omega}\to\Ga^{\infty}_{\omega}$ to be $[(a_n)]\mapsto [(\iota_n(a_n))]$ and $\pi^{\infty}_{\omega}: \Ga^{\infty}_{\omega}\to Q^{\infty}_{\omega}$ to be $[(\gamma_n)]\to[(\pi_n(\gamma_n))]$. We prove that
$$1\to N^{\infty}_{\omega}\xrightarrow{\iota^{\infty}_{\omega}}\Ga^{\infty}_{\omega}\xrightarrow{\pi^{\infty}_{\omega}} Q^{\infty}_{\omega}\to 1$$
is a short exact sequence.

First of all, it is direct to see that $\pi^{\infty}_{\omega}\circ\iota^{\infty}_{\omega}$ is the trivial map from $N^{\infty}_{\omega}$ onto the unit element $[(e_{Q_n})]\in Q^{\infty}_{\omega}$. To see $\iota^{\infty}_{\omega}$ is injective, let $[(a_n)]\in N^{\infty}_{\omega}$ such that $\iota^{\infty}_{\omega}([(a_n)])=[(e_{\Ga_n})]\in\Ga^{\infty}_{\omega}$. This means that
$$\{n\in\IN\mid \iota_n(a_n)=e_{\Ga_n}\}\in\omega.$$
Since each $\iota_n$ is injective, this means that $\{n\in\IN\mid a_n=e_{N_n}\}\in\omega$, i.e., $(a_n)\in I_{0,\omega}$. Similarly, one can also prove that $\pi^{\infty}_{\omega}$ is surjective. For the last step, take $[(\gamma_n)]$ to be such that $\pi^{\infty}_{\omega}([(\gamma_n)])=[(e_{Q_n})]$. Then
$$\{n\in\IN\mid \pi_n(\gamma_n)=e_{Q_n}\}\in\omega.$$
Since each $1\to N_n\to \Ga_n\to Q_n\to 1$ is exact, this means that $\{n\in\IN\mid \gamma_n\in Im(\iota_n)\}\in\omega$. This 
means that $[(\gamma_n)]\in Im(\iota^{\infty}_\omega)$.
\end{proof}

\subsection{Banach property (T) for Limit groups}\label{subsec: Banach T for limit groups}

For fixed ultrafilter $\omega\in\partial_\beta\IN$, we define $\IC_{u,\infty}[X]_{\omega}$ to be as follows. Recall that $\IC_u[X]$ is a subalgebra of $\prod_{n\in\IN}\IC_u[\Ga_n]$. We shall then write an element in $\IC_u[X]$ as a sequence $(T_n)$. Denote by
$$\I^{\infty}_{\omega}=\{(T_n)\in\IC_u[X]\mid \lim_{n\to\omega}\|T_n\|_{\ell^1}=0\}.$$
It is direct to see that $\I^\infty_\omega$ is a two-side ideal of $\IC_u[X]$. We define $\IC_{u,\infty}[X]_{\omega}$ to be the quotient algebra of $\IC_u[X]$ by the ideal $\I_{\omega}^\infty$, i.e., 
$$\IC_{u,\infty}[X]_{\omega}=\frac{\IC_u[X]}{\I^\infty_\omega}.$$
We denote the quotient map by $\pi^\infty_\omega:\IC_u[X]\to\IC_{u,\infty}[X]_\omega$. Restrict the map $\pi^\infty_\omega$ on $\ell^{\infty}(X)$, then the image $\pi^\infty_\omega(\ell^\infty(X))$ forms a $C^*$-algebra under the norm
$$\|\pi^{\infty}_\omega(f)\|=\lim_{n\to\omega}\|f_n\|_{\infty},$$
where each $f_n$ is the restriction of $f$ on $\Ga_n$. We denote the quotient $C^*$-algebra by $C(Y)$, where $Y$ is obtained by Gelfand transformation.

We define the $\Ga^\infty_{\omega}$-action on $\ell^{\infty}(X)/\I^\infty_\omega\cong C(Y)$ by
$$[\gamma_n]\cdot[f_n]=[\gamma_nf_n],$$
where $\Ga_n$ acts on $\ell^\infty(\Ga_n)$ by left regular action. One can check this action is well-defined. Indeed, if $(\gamma_n)\sim_{\omega}(\gamma'_n)$, then for any $(f_n)\in\prod_{n\in\IN}\ell^{\infty}(\Ga_n)$, we have that
$$\{n\in\IN\mid \gamma_nf_n=\gamma_n'f_n\}\in\omega.$$
Thus, $\lim_{n\to\omega}\|\gamma_nf_n-\gamma_n'f_n\|=0$. Thus, $Y$ is induced with a $\Ga^\infty_{\omega}$-action.

For any $n\in\IN$, consider $\xi_n\in\ell^2(X)$ to be the normalized characteristic function on $\Ga_n\subseteq X$. Then $\phi_n: \ell^{\infty}(X)\to\IC$ defined by
$$f\mapsto\langle f\xi_n,\xi_n\rangle$$
defines a positive linear function on $\ell^{\infty}(X)$. For the fixed free ultrafilter $\omega\in\partial_{\beta}\IN$, define $\phi$ to be
$$\phi_\omega(f)=\lim_{i\to\omega}\phi_{i}(f)$$
which defines an element in $\ell^{\infty}(X)'$. Moreover, note that $\phi_\omega$ descends to a state on $\ell^{\infty}(X)/\I^\infty_\omega$. Indeed, if $f\in\I^\infty_\omega$, then for any $\varepsilon>0$, there exists $D\in\omega$ such that $\sup_{x\in\Ga_n}|f(x)|\leq\varepsilon$ for all $n\in D$. As a result, for any $n\in D$, we conclude that
$$|\phi_n(f)|=\frac{|\sum_{x\in\Ga_n}f(x)|}{\#\Ga}\leq\varepsilon.$$
With a similar proof as above, one can also check that the state $\phi_\omega$ is $\Ga^\infty_\omega$-invariant. According to the Riesz Representation Theorem, there exists a $\Ga^\infty_\omega$-invariant probability measure $\mu_{\phi_\omega}$ on $Y$ such that
$$\phi(f)=\int_{Y}f(x)d\mu_{\phi_\omega}(x)$$
for any $f\in C(Y)$. We denote $Y\rtimes\Ga^{\infty}_{\omega}$ the transformation groupoid.

\begin{Lem}\label{lem: crossed product}
The algebra $\IC_{u,\infty}[X]_\omega$ is isomorphic to algebraic crossed product $C_c(Y\rtimes\Ga^{\infty}_{\omega})\subseteq C(Y)\rtimes \Ga^\infty_\omega$.
\end{Lem}

\begin{proof}
For any operator $T\in\IC_u[X]$, denote $R=\prop(T)$. Since the sequence $(\Ga_n)$ has uniform bounded geometry, there exists $N\in\IN$ such that $\#B(e_n,R)\leq N$. For any $n\in\IN$, we can label $B(e_n,R)$ with $\{1,\cdots, N\}$. By using a similar method as in Case 1 of Example \ref{exa: explanation of Lemma 3.1}, every operator $T\in\IC_u[X]$ can be written as a finite sum $T=\sum (f^{(n)}_{\gamma}\cdot\gamma^{(n)})$ with the sequence $(\gamma^{(n)})\in\prod^b_{n\in\IN}\Ga_n$ and $f^{(n)}_\gamma\in\ell^{\infty}(\Ga_n)$. Define the map
$$\Phi_\omega: \IC_u[X]\to C_c(Y\rtimes\Ga^{\infty}_{\omega})\quad\text{by}\quad T\mapsto \sum f_{\gamma}\cdot\gamma$$
where $\gamma=[\gamma^{(n)}]\in\Ga^\infty_\omega$ and $f_\gamma=\pi^\infty_\omega((f^{(n)}_\gamma))\in C(Y)$. Since $(\Ga_n)$ has uniform bounded geometry, the image of $T$ under $\Phi_\omega$ must be a finite sum. To see this map is well-defined, assume that there is another decomposition of $T=\sum (f^{(n)}_{\eta}\cdot\eta^{(n)})$ associated with another labelling. If $\eta=\gamma\in\Ga^\infty_\omega$, then by definition, we conclude that
$$D=\{n\in\IN\mid \eta_n=\gamma_n\}\in\omega.$$
From the construction of the decomposition, we have that $f^{(n)}_\gamma=f^{(n)}_\eta$ for any $n\in\IN$. As a result, $f_\gamma=f_\eta\in  C(Y)$. This shows that the map $\Phi_\omega$ is well-defined. From the construction of the map $\Phi_\omega$, it is also direct to see that $\Phi_\omega$ is a surjective homomorphism.

Moreover, if $T\in\IC_u[X]$ satisfies that $\Phi_\omega(T)=0$, then for a decomposition $T=\sum (f^{(n)}_{\gamma}\cdot\gamma^{(n)})$, we have that $\lim_{n\to\omega}\|f^{(n)}_\gamma\|=0$ for any $\gamma\in\Ga^{\infty}_\omega$. By uniform bounded geometry of $(\Ga_n)$, it implies that $\lim_{n\to\omega}\|T_n\|_{\ell^1}=0$. This means that the kernel of $\Phi_\omega$ is exactly equal to $\I^\infty_\omega$, i.e., $\IC_{u,\infty}[X]_\omega\cong C_c(Y\rtimes\Ga^{\infty}_{\omega})$.
\end{proof}

\begin{Pro}\label{pro: banach property T for limit group}
Let $(\Ga_n)_{n\in\IN}$ be a sequence of finite groups and $X=\bigsqcup_{n\in\IN}$ the separated disjoint union of $(\Ga_n)$. For any uniformly convex Banach space $\fB$ and $p\in(1,\infty)$, if $X$ has (uniform) geometric property $(T_{L^p(\mu_{\phi},\fB)})$, then $\Ga^\infty_{\omega}$ has (uniform) property $(T_{\fB})$.
\end{Pro}

\begin{proof}
We shall only prove the uniform case of this proposition, the proof for the regular case is somehow parallel to the uniform case.

By Lemma \ref{lem: crossed product}, we denote $\iota: \IC\Ga^\infty_\omega\to\IC_{u,\infty}[X]_{\omega}\cong C_c(\Ga^\infty_\omega, C(Y))$ the canonical inclusion induced by the constant inclusion $\IC\to C(Y)$. For any representation $\rho:\IC\Ga^\infty_\omega\to\L(\fB)$ and $p\in(1,\infty)$, we claim that there exists a representation $\pi:\IC_{u,\infty}[X]_\omega\to\L(L^p(\mu_{\phi},\fB))$ such that $\pi\circ\iota: \IC\Ga^\infty_\omega\to\IC_{u,\infty}[X]_{\omega}\to\L(L^p(\mu_{\phi},\fB))$ contains $\rho:\IC\Ga^\infty_\omega\to\L(\fB)$ as a subrepresentation.

Define $M: C(Y)\to\L(L^p(\mu_{\phi_\omega}))$ to be the multiplication representation, i.e.,
$$(M_f(\xi))(x)=f(x)\xi(x),$$
for any $f\in C(Y)$ and $\xi\in L^p(\mu_{\phi_\omega})$. Notice that $L^p(\mu_{\phi_\omega})$ also admits a canonical $\Ga^\infty_\omega$ representation by
$$(V_\gamma\xi)(x)=\xi(\gamma^{-1}x)$$
for any $\gamma\in \Ga^\infty_\omega$ and $\xi\in L^p(\mu_{\phi_{\omega}})$. One can check that $V$ and $M$ are coherent, i.e., $V_\gamma M_fV^{-1}_{\gamma}=M_{\gamma f}$. Since $C_c(Y\rtimes\Ga^{\infty}_{\omega})$ is generated by $\IC\Ga^{\infty}_{\omega}$ and $C(Y)$, the two coherent representations $V\ox\rho: \IC\Ga^{\infty}_{\omega}\to\L(L^p(\mu_{\phi_{\omega}},\fB))$ and $M\ox 1:C(Y)\to \L(L^p(\mu_{\phi_{\omega}},\fB))$ defines a representation
$$\pi: C_c(Y\rtimes\Ga^{\infty}_{\omega})\to \L(L^p(\mu_{\phi_{\omega}},\fB)).$$
Since $C_c(Y\rtimes\Ga^{\infty}_{\omega})$ is a quotient of $\IC_u[X]$, $\pi$ lifts a representation $\IC_u[X]\to\L(L^p(\mu_{\phi_{\omega}},\fB))$, a little abuse of notation, we shall still denote this representation $\pi$. It is direct to see that $\pi$ is contractive. One can easily check that $\pi\circ\iota: \IC\Ga^\infty_{\omega}\to\IC_{u,\infty}[X]_\omega\to\L(L^p(\mu_{\phi_\omega},\fB))$ is equal to $V\ox\rho$. Since $\mu_{\phi_\omega}$ is a probability measure, the inclusion map
$$I:\fB\to L^p(\mu_{\phi_\omega},\fB),\qquad \xi\mapsto 1_{\mu_{\phi_\omega}}\ox\xi$$
gives an isometry, where $1_{\mu_{\phi_\omega}}$ is the constant function. Moreover, it is direct to check that $I$ is $\Ga^\infty_\omega$-equivariant, where $\fB$ is equipped with $\rho(\Ga^{\infty}_\omega)$-action and $L^p(\mu_{\phi_\omega},\fB)$ is equipped with $V\ox\rho(\Ga^\infty_\omega)$-action. Thus $\rho$ can be seen as a subrepresentation of $\pi\circ\iota$ and $I$ takes $\fB^\rho$ into $L^p(\mu_{\phi_\omega},\fB)^{\pi\circ\iota}$ by definition. Thus, the map $I$ induces an isometrical embedding by definition
$$\wh I:\fB/\fB^{\rho}\to L^p(\mu_{\phi_\omega},\fB)/L^p(\mu_{\phi_\omega},\fB)^{\pi\circ\iota}.$$
By Lemma \ref{lem: complemented representation}, $\wh I$ induces a continuous embedding $\wt I:\fB_\rho\to L^p(\mu_{\phi_\omega},\fB)_{\pi\circ\iota}$.

Now, assume for a contradiction that $\Ga^\infty_\omega$ does not have uniform property ($T_{\fB}$). Then for any $\varepsilon>0$, there exists a representation $\rho:\IC\Ga^\infty_\omega\to\L(\fB)$ and $\xi\in\fB_{\rho}$ with $\|\xi\|=1$ such that for any $\gamma\in S$, we have that
$$\|\rho(\gamma)\xi-\xi\|\leq\varepsilon.$$
By the construction above, for any such representation $\rho$, we take $\pi: \IC_{u,\infty}[X]_\omega\to\L(L^p(\mu_{\phi_\omega},\fB))$ such that $\pi\circ\iota$ contains $\rho$ as a subrepresentation. A little abuse of notation, we shall still denote it by $\pi$ the representation
$$\IC_u[X]\xrightarrow{\text{quotient map}}\IC_{u,\infty}[X]_\omega\to\L(L^p(\mu_{\phi_\omega},\fB))$$
Since $\rho$ is a subrepresentation of $\pi\circ\iota$, we conclude that the representation
$$\pi\circ\iota: \IC\Ga^\infty_{\omega}\to\IC_{u,\infty}[X]_\omega\to\L(L^p(\mu_{\phi_\omega},\fB))$$
also satisfies that for any $\gamma\in S^{\infty}_\omega$, by \cite[Proposition 2.10]{BFGM2007}
$$\|\pi\circ\iota(\gamma)[1_{\mu_{\phi_{\omega}}}\ox\xi]-[1_{\mu_{\phi_{\omega}}}\ox\xi]\|_{L^p(\mu_{\phi_\omega},\fB)/L^p(\mu_{\phi_\omega},\fB)^{\pi\circ\iota}}\leq \varepsilon.$$
By Lemma \ref{lem: invariant vector in group language}, the invariant space $\fB^{\pi\circ\iota}$ coincides with $\fB^{\pi}$. Thus, the space $\fB/\fB^\rho$ is isomorphic to $(1\ox\fB+L^p(\mu_{\phi_\omega},\fB)^{\pi})/L^p(\mu_{\phi_\omega},\fB)^{\pi}$. Then$[1_{\mu_{\phi_\omega}}\ox\xi]\in L^p(\mu_{\phi_\omega},\fB)/L^p(\mu_{\phi_\omega},\fB)^{\pi}$ also has norm $r_0$. For any $\gamma\in S^\infty_\omega$, one can always take $(\gamma^{(n)})\in \prod_{n\in\IN}S_n$ such that $\gamma=[\gamma^{(n)}]$. Thus 
\[\begin{split}
&\|\pi((\gamma^{(n)}))[1_{\mu_{\phi_\omega}}\ox\xi]-[1_{\mu_{\phi_\omega}}\ox\xi]\|_{L^p(\mu_{\phi_\omega},\fB)/L^p(\mu_{\phi_\omega},\fB)^{\pi}}\\
=&\|\pi(\iota(\gamma))[1_{\mu_{\phi_\omega}}\ox\xi]-[1_{\mu_{\phi_\omega}}\ox\xi]\|_{L^p(\mu_{\phi_\omega},\fB)/L^p(\mu_{\phi_\omega},\fB)^{\pi}}\\
=&\|\pi\circ\iota(\gamma)[1_{\mu_{\phi_{\omega}}}\ox\xi]-[1_{\mu_{\phi_{\omega}}}\ox\xi]\|_{L^p(\mu_{\phi_\omega},\fB)/L^p(\mu_{\phi_\omega},\fB)^{\pi\circ\iota}}\leq \varepsilon.
\end{split}\]
Since the canonical quotient map $\prod_{n\in\IN}S_n\to S^\infty_{\omega}$ is a surjection, again by Lemma \ref{lem: invariant vector in group language} and Case 1 in Example \ref{exa: explanation of Lemma 3.1}, for any partial translation $V\in\IC_u[X]$ with $\supp(V)\subseteq E_0$, one can write it as $V=\sum\chi_{B_{V,i}}\cdot A_i$, where $A_i=(s^{(n)}_i)\in\prod S_n$. By the construction, one also concludes that $\{B_{V,i}\}$ is pairwise-disjoint and $\sum_{i=1}^n\chi_{B_{V,i}}=\Phi(V)$. For each $i$, we denote $s_i=[s^{(n)}_i]\in\Ga^\infty_\omega$ the element in the limit group. Moreover, $\{s_i\}$ is equal to the generating set $S^\infty_\omega$ by definition. We finally conclude that
\[\begin{split}\|\pi(V)[1_{\mu_{\phi_\omega}}\ox\xi]-\pi(\Phi(V))[1_{\mu_{\phi_\omega}}\ox\xi]\|&=\left\|\sum\pi(\chi_{B_{V,i}}\cdot A_i)[1_{\mu_{\phi_\omega}}\ox\xi]-\sum\pi(\chi_{B_{V,i}})[1_{\mu_{\phi_\omega}}\ox\xi]\right\|\\
&\leq\sum_{s_i\in S^\infty_\omega}\left\|\pi(\chi_{B_{V,i}})\left(\pi(\iota(s_i))[1_{\mu_{\phi_\omega}}\ox\xi]-[1_{\mu_{\phi_\omega}}\ox\xi]\right)\right\|\\
&\leq\sum_{s_i\in S^\infty_\omega}\left\|\left(\pi(\iota(s_i))[1_{\mu_{\phi_\omega}}\ox\xi]-[1_{\mu_{\phi_\omega}}\ox\xi]\right)\right\|\leq\#S^\infty_\omega\cdot\varepsilon.\end{split}\]
Here the norm is taken in the quotient space $L^p(\mu_{\phi_\omega},\fB)/L^p(\mu_{\phi_\omega},\fB)^{\pi}$ and $\|\pi(\chi_{B_{V,i}})\|\leq 1$ because $\pi$ is assumed to be contractive. Since $\varepsilon$ is arbitrarily taken, this contradicts to $\IC_u[X]$ has uniform geometric property ($T_{L^p(\mu_{\phi},\fB)}$).
\end{proof}

\subsection{Residually finite groups and box spaces}\label{subsec: residually finite group}

Let $\Ga$ be a residually finite group. A \emph{filtration} is a nested sequence of finite index normal subgroups $\{N_n\}_{n\in\IN}$ of $\Ga$
$$\Ga \trianglerighteq N_1 \trianglerighteq N_2 \trianglerighteq\cdots \trianglerighteq N_n \trianglerighteq\cdots,$$
such that $\bigcap_{n\in\IN}N_n=\{e\}$. The \emph{box space} of $\Ga$ associated with this filtration is defined to be the disjoint union $\Box_{\{N_n\}}(\Ga)=\bigsqcup_{n\in\IN}\Ga/N_n$ equipped with the metric $d$ satisfying that $d$ is equal to the quotient metric on each $\Ga/N_n$ and $d(\Ga/N_n,\Ga/N_m)=\infty$ whenever $n\ne m$. If we fix $S$ a symmetric generating set of $\Ga$, then
$$E_0=\{(x,sx)\mid x\in \Ga/N_n, s\in S\}$$
will automatically form a generating set of the coarse structure of $X$.

\begin{Pro}\label{pro: example of box spaces}
Let $\Ga$ be a finitely generated, residually finite group with (uniform) property ($T_{\fB}$), then any box space $\Box_{\{\Ga_n\}}(\Ga)$ of $\Ga$ has (uniform) geometric property ($T_\fB$)
\end{Pro}

Before we can prove Proposition \ref{pro: example of box spaces}, we shall need the some preparation. For simplicity, we shall denote $X_n=\Ga/\Ga_n$ and $X=\Box_{\{\Ga_n\}}(\Ga)$. Denote by $\IC\Ga$ the group algebra of the group $\Ga$. Then a linear isometric representation of $\Ga$ is equivalent to an isometric representation of the algebra $\IC\Ga$. For each $X_n$, there is a canonical map
$$\iota_n:\IC\Ga\xrightarrow{\pi_n}\IC X_n\xrightarrow{i}\IC_u[X_n],$$
where $\pi_n$ is induced by the quotient homomorphism $\pi_n:\Ga\to\Ga/\Ga_n$ and $i:\IC X_n\to\IC_u[X_n]$ is the canonical inclusion. Since $\IC_u[X]$ is a subalgebra of $\prod_{n\in\IN}\IC_u[X_n]$, it is direct to see there is a canonical map $\iota: \IC\Ga\to\IC_u[X]$ which composes with the canonical projection on $\IC_u[X_n]$ is exactly $\iota_n$. Moreover, since $\bigcap\Ga_n=\{e\}$, the map $\iota$ is an injection. One should treat the following lemma as a uniform version of Lemma \ref{lem: invariant vector in group language}.

\begin{Lem}\label{lem: invariant spaces coincide}\begin{itemize}
\item[(1)] For any partial translation $V\in\IC_u[X]$ with $\supp(V)\subseteq E_0$, there exist a finite family $\{A_{V,s}\}_{s\in S}$ of subsets of $X$ such that $V=\sum_{s\in S}\chi_{A_{V,s}}\cdot\iota(s)$.
\item[(2)] The invariant space $\fB^{\rho}$ is equal to the invariant space $\fB^{\pi}$.\qed
\end{itemize}\end{Lem}

\begin{proof}
(1) For any $\xi\in\fB^{\pi}$, since $\iota(s)$ is a partial translation in $\IC_u[X]$, it is direct to see that $\xi\in\fB^{\rho}$. On the other hand, fix a vector $\xi\in\fB^{\rho}$.  Let $V\in\IC_u[X]$ be a partial translation such that $\supp(V)\subseteq E_0$. For any generator $s\in S\subseteq \Ga$, we denote $A_{V,s}=P_1(\supp(\iota(s))\cap\supp(V))$, where $P_1: X\times X\to X$ is the projection onto the first coordinary, i.e., $P_1(x_1,x_2)\mapsto x_1$. By definition, one has that $\chi_{A_{V,s}}\cdot V=\chi_{A_{V,s}}\cdot\iota(s)$. Since $\bigcup_{s\in S}\supp(\iota(s))=E_0$, we conclude that $\bigsqcup_{s\in S}A_{V,s}=P_1(\supp(V))$. Denote by $\chi_{A_{V,s}}$ the characteristic function on $A_{V,s}$. We then have that
$$V=\Phi(V)\cdot V=\chi_{P_1(\supp(V))}\cdot V=\chi_{\bigsqcup_{s\in S}A_{V,s}}\cdot V=\sum_{s\in S}\chi_{A_{V,s}}\cdot V=\sum_{s\in S}\chi_{A_{V,s}}\cdot\iota(s).$$

(2) It is direct to see that $\fB^{\pi}\subseteq\fB^{\rho}$ since $\iota(s)$ is a partial translation for any $s\in S$. On the other hand, for any $\xi\in\fB^{\rho}$, for any partial translation $V\in\IC_u[X]$, one has that $\Phi(V)\cdot V=V$. Write $V=\sum_{s\in S}\chi_{A_{V,s}}\cdot\iota(s)$. We have that
\[\begin{split}\pi(V)\xi&=\sum_{s\in S}\pi(\chi_{A_{V,s}})\cdot\pi(\iota(s))\xi=\sum_{s\in S}\pi(\chi_{A_{V,s}})\cdot\rho(s)\xi\\
&=\sum_{s\in S}\pi(\chi_{A_{V,s}})\xi=\pi(\chi_{P_1(\supp(V))})\xi=\Phi(V)\xi.
\end{split}\]
This proves that $\xi\in\fB^{\pi}$.
\end{proof}

\begin{proof}[Proof of Proposition \ref{pro: example of box spaces}]
We still only prove the uniform version. Assume that $X$ has no uniform geometric property ($T_{\fB}$). There exists a representation $\pi: \IC_u[X]\to \L(\fB)$ such that the induced representation $\wt\pi:\IC_u[X]\to \L(\fB/\fB^{\pi})$ has $\varepsilon$-invariant vectors for any $\varepsilon>0$. By definition, we conclude that for any $\varepsilon>0$, there exists $[\xi]\in\fB/\fB^{\pi}$ with norm $1$ such that
$$\|\wt\pi(V)[\xi]-\wt\pi(\Phi(V))[\xi]\|\leq\varepsilon$$
for any partial translation whose graph is in $E_0\in \E$. Denote by $\rho$ the induced representation of $\pi$ by composing with $\iota$. Then for any $s\in S$, by Lemma \ref{lem: invariant spaces coincide}, one then has that
$$\|\wt{\rho}(s)[\xi]-[\xi]\|=\|\wt\pi(\iota(s))[\xi]-[\xi]\|\leq\varepsilon$$
for $[\xi]\in\fB/\fB^{\pi}=\fB/\fB^{\rho}$. This shows that $\rho: G\to O(\fB/\fB^{\rho})$ has $\varepsilon$-invariant vector, which means $\Ga$ does not have uniform property ($T_{\fB}$).
\end{proof}

Proposition \ref{pro: example of box spaces} provides a great number of examples of metric spaces with property ($T_{\fB}$). For example, by \cite[Theorem 1.1]{BFGM2007} and \cite{KazhdanT}, for $k\geq 3$, we concludes that any box space of $SL_k(\IZ)$ has (uniform) geometric property ($T_{L^p(\mu)}$) for any $\sigma$-unital measure $\mu$ and $1\leq p<\infty$. More strongly, combining the results proved by Oppenheimer in \cite{Oppen2023} and the results in \cite{LS2023} by T.~de~Laat and M.~de~la~Salle, we further deduce that for $k\geq 3$, any box space of $SL_k(\IZ)$ possesses property ($T_{\fB}$) with respect to any super-reflexive Banach space $\fB$.

On the other hand, it is proved in \cite[Example 2.11]{GQW2024} that all limit groups of $\Box(\Ga)$ is canonically isomorphic to $\Ga$ itself. By Proposition \ref{pro: banach property T for limit group}, we then have the inverse of Proposition \ref{pro: example of box spaces}. Let $\phi_n$ be the state of $\ell^\infty(X)$ defined as in the last section, and $\phi$ any cluster point of the sequence $\{\phi_n\}_{n\in\IN}$ in $\ell^{\infty}(X)'$. It descends to a positive functional on $\ell^{\infty}(X)/C_0(X)\cong C(\partial_\beta X)$. Then $\phi$ defines a $\Ga$-invariant state $\phi: C(\partial_\beta X)\to\IC$. By Riesz representation theorem, there exists a $\Ga$-invariant measure $\mu_{\phi}$ on $\partial_{\beta}X$ such that $\phi(f)=\int_{\partial_\beta X}f(x)d\mu_\phi(x)$. As a direct corollary of Proposition \ref{pro: banach property T for limit group}, we have the following result.

\begin{Pro}\label{pro: from box spaces to group}
With the notation as above, for any uniformly convex Banach space $\fB$ and $p\in(1,\infty)$ and any cluster $\phi\in\ell^\infty(X)'$, if $X$ has (uniform) geometric property $(T_{L^p(\mu_{\phi},\fB)})$, then $\Ga$ has (uniform) property $(T_{\fB})$.\qed
\end{Pro}

As a direct corollary of Proposition \ref{pro: example of box spaces} and Proposition \ref{pro: from box spaces to group}, we have the following theorem.

\begin{Thm}\label{thm: box iff group in lp}
Let $\Ga$ be a finitely generated, residually finite group. For any $p\in(1,\infty)$, the following are equivalent:\begin{itemize}
\item[(1)] $\Ga$ has property ($T_{\sL^p}$);
\item[(2)] for any filtration $\{\Ga_n\}$, $\Box_{\{\Ga_n\}}(\Ga)$ has geometric property ($T_{\sL^p}$);
\item[(3)] there exists a filtration $\{\Ga_n\}$ such that $\Box_{\{\Ga_n\}}(\Ga)$ has geometric property ($T_{\sL^p}$).
\end{itemize}\end{Thm}

\begin{proof}
By Proposition \ref{pro: example of box spaces}, one has that $(1)\Rightarrow (2)\Rightarrow (3)$. For any measure spaces $(X,\mu)$ and $(Y,\nu)$, one has that $L^p(X,\mu,L^p(Y,\nu))$ is isometric to $L^p(X\times Y,\mu\times\nu)$, which is also an $L^p$-space. 
From the proof of Proposition \ref{pro: from box spaces to group}, one can also conclude that the spectral gap for property ($T_{\fB}$) for $\Ga$ is greater than $\frac{1}{\#S}$ times that for property $(T_{L^p(\mu_{\phi},\fB)})$ for $X$. Then $(3)\Rightarrow (1)$ follows directly from Proposition \ref{pro: from box spaces to group}.
\end{proof}

When $p=2$, Theorem \ref{thm: box iff group in lp} provides an alternative approach to the proof of \cite[Theorem 7.1]{WY2014} without using the spectral criterion of property (T).

\subsection{A geometric description of Lafforgue's strong Banach property (T)}

In this section, we will discuss \emph{strong Banach property (T)}. Strong (Banach) property (T) was introduced by V.~Lafforgue in \cite{Lafforgue2008, Lafforgue2010}. For groups possessing this property, Lafforgue's approach to the Baum-Connes conjecture via Banach $KK$-theory is no longer applicable, see \cite{Lafforgue2002}. It is proved by M.~de~la~Salle that every lattice in a higher-rank group has strong property (T), consequently yielding an extensive collection of groups demonstrating this property, see \cite{dlS2019}. In this section, we shall introduce a notion of geometric strong Banach property (T) as a geometric counterpart of strong Banach property (T). 

In the following of this subsection, we shall not assume a representation to be isometric. Let $\sB$ be a family of Banach spaces closed under duality and complex conjugate. According to the framework established in \cite{dlS2016}, it is common to additionally assume that the type of $\sB$ is greater than $1$. However, in this section, we only assume that $\sB$ satisfies one of the following conditions:\begin{itemize}
\item[(1)] $\sB=\sH$ is the family of all Hilbert spaces;
\item[(2)] $\sB=\sL^{p,q}$ is the family of all $L^p$ and $L^q$ spaces with $1<p,q<\infty$ and $\frac 1p+\frac 1q=1$;
\item[(3)] $\sB$ is a uniformly convex family of Banach spaces, additionally closed under duality, conjugation, ultraproduct and $L^2$-Lebesgue-Bochner tensor product, i.e., $\fB\in\sB\Rightarrow L^2(\mu,\fB)\in\sB$ for any measure space $(X,\mu)$.
\end{itemize}
Let $\Ga$ be a countable discrete group. Fix a length function $\ell$ on $\Ga$. A representation $\pi: \Ga\to\L(\fB)$ is said to have \emph{$(\ell,s,c)$-small exponential growth}, if $\|\pi(\gamma)\|_{\fB}\leq c\cdot e^{s\ell(\gamma)}$ for every $\gamma\in\Ga$, where $s,c>0$. We define the following norm on the group algebra $C_c(\Ga)$:
$$\|f\|_{\ell,s,c}=\sup\{\|\pi(f)\|\mid \pi\text{ has }(\ell,s,c)\text{-small exponential growth}\}.$$
The completion of $C_c(\Ga)$ under $\|\cdot\|_{\ell,s,c}$ is denoted by $\C_{\ell,s,c}(\Ga)$. If $s=c=0$, then $\|\cdot\|_{\ell,0,0}$ is equal to $\|\cdot\|_{\sB,\max}$. Since $\sB$ is closed under duality, there exists an involution map in $\C_{\ell,s,c}(X)$, denoted by $T\mapsto T^*$, which is an isometric map.

We shall first recall the definition of the strong Banach property (T) for groups.

\begin{Def}
A countable discrete group $\Ga$ has \emph{strong Banach property (T)} associated with $\sB$ if for every length function $\ell$, there exists $s>0$ such that for any $c>0$, the Banach algebra $\C_{\ell,s,c}(\Ga)$ has a \emph{Kazhdan projection}, i.e., a selfadjoint idempotent $p$ such that $\pi(p)$ is a projection on the space of invariant vectors for every representation $\pi$ with $(s,c)$-small exponential growth.
\end{Def}

Now, let us shift our focus back to metric spaces. Consider $X$ as a discrete metric space with bounded geometry, and assume that its coarse structure is monogenic. Set $\E$ to be the coarse structure of $X$ and $E_0$ is the generator of $\E$. A metric $\delta$ on $X$ is said \emph{finer} than $d$ if the identity map $id: (X,d)\to (X,\delta)$ is bornologous, i.e., any controlled set in $\E_d$ must belong to $\E_{\delta}$. Let $\D$ be a set of metrics on $X$ which are finer than $d$. For any $E\in\E_d$, we denote the \emph{$\delta$-propagation} of $E$ to be
$$\prop_{\delta}(E)=\sup\{\delta(x,y)\mid (x,y)\in E\}.$$
For $T\in\IC_u[X]$, we shall direct denote $\prop_{\delta}(T)=\prop_{\delta}(\supp(T))$. Since $\delta$ is finer than $d$, the $\delta$-propagation of $T\in\IC_u[X]$ must be finite. For $s,c>0$, a representation $\pi:\IC_u[X]\to\L(\fB)$ is said to have $(\delta,s,c)$-small exponential growth, if $\|\pi(V)\|_{\fB}\leq c\cdot e^{s\cdot\prop_{\delta}(V)}$ for any partial translation $V\in\IC_u[X]$. We define the following norm on $\IC_u[X]$:
$$\|T\|_{\ell,s,c}=\sup\{\|\pi(T)\|\mid \pi\text{ has }(\delta,s,c)\text{-small exponential growth}\}.$$
The completion of $\IC_u[X]$ under $\|\cdot\|_{\delta,s,c}$ is denoted by $\C_{\delta,s,c}(X)$. Similarly, there is an isometric involution map on $\C_{\delta,s,c}(X)$. For every $(\delta,s,c)$-representation $\pi$, it is direct to see that $\pi$ extends to a representation of $\C_{\delta,s,c}(X)$ by the universal property.

\begin{Def}
Let $(X,d)$ be a space, $\D$ a set of metrics on $X$ finer than $d$. $X$ has \emph{geometric strong Banach property (T)} associated with $\D$ and $\sB$ if, for every metric $\delta\in\D$, there exist $s>0$ such that for any $c>0$, the Banach algebra $\C_{\delta,s,c}(X)$ has a \emph{Kazhdan projection}, i.e., a selfadjoint idempotent $p$ such that $\pi(p)$ is a projection on the space of invariant vectors for every representation $\pi$ with $(\delta,s,c)$-small exponential growth.
\end{Def}

\begin{Rem}
Strong Banach property (T) implies Banach property (T) by taking the length function $\ell$ to be a bounded function. Parallelly, take $\D$ to be the set of all bounded metrics on $X$, then geometric strong Banach property (T) of $X$ associated with $\sB$ and $\D$ is equivalent to the geometric property ($T_{\sB}$).
\end{Rem}

\begin{Lem}
If two metric $\delta_1$ and $\delta_2$ are quasi-isometric, then $X$ has geometric strong Banach property ($T$) associated with $\sB$ and $\delta_1$ if and only if $X$ has geometric strong Banach property ($T$) associated with $\sB$ and $\delta_2$.
\end{Lem}

\begin{proof}
By definition, there exists $L,C>0$ such that for any $T\in\IC_u[X]$, one has that
$$\frac 1L\cdot\prop_{\delta_1}(T)-C\leq\prop_{\delta_2}(T)\leq L\cdot\prop_{\delta_1}(T)+C.$$
Assume that $X$ has geometric strong Banach property ($T$) associated with $\sB$ and $\delta_1$. Then there exists $s>0$ such that for any $c>0$, the Banach algebra $\C_{\delta_1,s,c}(X)$ has a Kazhdan projection. Define $s‘=\frac{s}{L}$. Then for any $c>0$, we conclude that
$$s'\prop_{\delta_2}(T)+\ln(c)\leq s\prop_{\delta_1}(T)+\ln(c)+C$$
Thus a representation with $(\delta_2,s',c)$-small exponential growth must have $(\delta_1,s,c+C)$-small exponential growth. Then the canonical quotient map $\C_{\delta_1,s,ce^C}(X)\to \C_{\delta_2,s',c}(X)$ will send the Kazhdan projection to a Kazhdan projection.
\end{proof}

Let $\Ga$ be a countable, discrete, residually finite group, $\ell$ a proper length function on $\Ga$. A proper length function $\ell$, to some extent, determines the coarsest metric on the group. For any length function $\ell'$, the identity map $(\Ga,\ell)\to (\Ga,\ell')$ must be bornologous. Let $\Box(\Ga)$ be a box space of $\Ga$, equipped with a quotient metric induced from the metric on $\Ga$ determined by a proper length function, denoted by $d$. Then any length function on $\Ga$ will give a metric on $\Box(\Ga)$ which is finer than $d$. We shall denote by $\D_{\ell}$ the set of all metric on $\Box(\Ga)$ which is determined by a length function on $\Ga$. Now, we are ready to claim

\begin{Thm}\label{thm: strong property (T)}
Let $\sB$ be a uniformly convex family of Banach spaces, additionally closed under duality, conjugation, ultraproduct and $L^2$-Lebesgue-Bochner tensor product, $\Ga$ a countable, discrete, residually finite group. Then the following are equivalent:\begin{itemize}
\item[(1)] $\Ga$ has strong Banach property (T) associated with $\sB$;
\item[(2)] all box spaces of $\Ga$ have geometric strong Banach property (T) associated with $\sB$ and $\D_\ell$;
\item[(3)] there exists a box space of $\Ga$ which has geometric strong Banach property (T) associated with $\sB$ and $\D_\ell$.
\end{itemize}\end{Thm}

\begin{proof}
The proof is a combination of the proof of Theorem \ref{thm: box iff group in lp} and Theorem \ref{thm: property T via Kazhdan projection}. Denote $X=\Box(\Ga)$. For any length function $\ell$ on $\Ga$, we shall denote the induced metric on $X$ to be $\delta_\ell$. By a similar argument with Proposition \ref{pro: example of box spaces}, the canonical inclusion $\iota:\IC\Ga\to\IC_u[X]$ extends to a homomorphism $\iota: \C_{\ell,s,c}(\Ga)\to\C_{\delta_{\ell},s,c}(X)$. Moreover, since $\sB$ is closed under $L^2$-Lebesgue-Bochner tensor product, with a similar construction with Proposition \ref{pro: from box spaces to group}, one has that the homomorphism $\iota$ is an isometry. Indeed, it is direct to see that $\|a\|\geq\|\iota(a)\|$ for any $a\in\IC\Ga$ since any representation of $\IC_u[X]$ with $(\delta_\ell,s,c)$-small exponential growth is a representation of $\IC\Ga$ with $(\ell,s,c)$-small exponential growth. For the other side, by using the construction in Proposition \ref{pro: from box spaces to group}, for any representation $\rho:\IC\Ga\to\L(\fB)$ with $(\ell,s,c)$-small exponential growth, one can construct a representation $\pi: \IC_u[X]\to\L(\fB')$ with $(\delta_\ell,s,c)$-small exponential growth for some $\fB'\in\sB$ such that $\fB\subseteq \fB'$ is a subspace and $\rho=\pi\circ\iota$ when restricting on $\fB$. This shows that $\|a\|\leq\|\iota(a)\|$ for any $a\in\IC\Ga$. Thus the Kazhdan projection $p\in\C_{\ell,s,c}(\Ga)$ will send to a Kazhdan projection $\iota(p)\in\C_{\delta_\ell,s,c}(X)$. This proves $(1)\Rightarrow (2)\Rightarrow (3)$.

For $(3)\Rightarrow (1)$, it is proved in Theorem \ref{thm: property T via Kazhdan projection} that the Kazhdan Projection must be the limit of $A^n$. This still holds for representations with $(\ell,s,c)$-small exponential growth by the same proof. Set $S\subseteq\Ga$, by Lemma \ref{lem: invariant spaces coincide}, one can directly take
$$A=\frac{1}{\#S}\cdot\sum_{s\in S}\frac{1+\iota(s)}2.$$
Notice that $A$ is in the image of $\iota$, i.e., it is the image of $a=\frac{1}{\#S}\cdot\sum_{s\in S}\frac{1+s}2\in\IC\Ga$. Since $\iota$ extends to an isometry, thus $a^k$ converges to $p\in \C_{\ell,s,c}(\Ga)$ if and only $A^k$ converges to $P\in\C_{\delta_\ell,s,c}(X)$. By \cite[Theorem 4.4]{DN2019}, the limit of $a^k$ is exactly the Kazhdan projection in $\C_{\ell,s,c}(\Ga)$. This finishes the proof.
\end{proof}

\subsection{FCE-by-FCE is incompatible with geometric property (T)}

As a corollary of the discussion before, in this subsection, we shall prove the following result.

\begin{Thm}\label{thm: FCE-by-FCE violates (T)}
Let $(1\to N_n\to \Ga_n\to Q_n\to 1)_{n\in\IN}$ be a sequence of uniformly finite generated finite group extensions which admits an ``FCE-by-FCE'' structure. If $(\Ga_n)$ is unbounded, then the sequence $(\Ga_n)_{n\in\IN}$ can not have geometric property (T).
\end{Thm}

Recall that a sequence of group extension $(1\to N_n\to \Ga_n\to Q_n\to 1)_{n\in\IN}$ admits an \emph{FCE-by-FCE} structure if the sequence of normal subgroups and quotient groups both admit \emph{fibred coarse embedding into Hilbert space}. The reader is referred to \cite[Definition 2.1]{CWY2013} for the definition of fibred coarse embedding into Hilbert space, we shall also recall its definition in the proof of Lemma \ref{lem: property for limit group}. Spaces with an \emph{FCE-by-FCE} structure are first introduced in \cite{DGWY2025} and it is proved that the coarse Novikov conjecture holds for such spaces. However, whether the maximal coarse Baum-Connes conjecture holds for such spaces is still unknown. It is only known that the maximal coarse Baum-Connes conjecture holds if one strengthens the condition to \emph{A-by-FCE}, see \cite{GWZ2024}. Since geometric property (T) is an obstruction to the maximal coarse Baum-Connes conjecture, thus it is natural to ask whether FCE-by-FCE structure is compatible with Geometric property (T). We answer this question nagatively by Theorem \ref{thm: FCE-by-FCE violates (T)}.

\begin{Lem}\label{lem: property for limit group}
Let $(\Ga_n)_{n\in\IN}$ be a sequence of Cayley graphs with uniformly finite degree.\begin{itemize}
\item[(1)] If $(\Ga_n)_{n\in\IN}$ admits a fibred coarse embedding into Hilbert space, then $\Ga^\infty_\omega$ is a-T-menable.
\item[(2)] If $(\Ga_n)_{n\in\IN}$ has geometric property (T), then $\Ga^\infty_\omega$ has property (T).
\end{itemize}\end{Lem}

\begin{proof}
\textbf{For (1)}. Assume that $\Ga_n$ admits a fibred coarse embedding into Hilbert space. Recall the definition of fibred coarse embedding, there exists\begin{itemize}
\item a field of Hilbert spaces $(H_g)_{g\in\Ga_n,n\in\IN}$;
\item a section $s:\Ga_n\to\sqcup_{g\in\Ga_n}H_g$ for all $n\in\IN$;
\item two non-decreasing functions $\rho_+$ and $\rho_-$ from $[0,\infty)$ to $[0,\infty)$ with $\lim_{r\to\infty}\rho_{\pm}(r)=\infty$;
\item a non-decreasing sequence of numbers $0\leq l_0\leq l_1\leq\cdots\leq l_n\leq\cdots$ with $\lim_{n\to\infty}l_n = \infty$.
\end{itemize}
such that for each $g\in \Ga_n$ there exists a "trivialization"
$$t_g:(H_h)_{h\in B_{\Ga_n}(g,l_n)}\to B_{\Ga_n}(g,l_n)\times H$$
such that the restriction of $t_g$ to the fiber $H_h$ for any $h\in B_{\Ga_n}(g,l_n)$ is an affine isometry $t_g(h):H_h\to H$, satisfying:
\begin{itemize}
\item [(1)]for any $h,h'\in B_{\Ga_n}(g,l_n)$,
$$\rho_-(d(h,h'))\leq \|t_g(h)(s(h))-t_g(h')(s(h'))\|\leq \rho_+(d(h,h'));$$
\item [(2)]for any $g,h\subset B_{\Ga_n}(g,l_n)\cap B_{\Ga_n}(h,l_n)\ne \emptyset$, there exists an affine isometry $t_{gh}:H\to H$ such that $t_g(k)\circ t^{-1}_h(k)= t_{xy}$ for all $k\in B_{\Ga_n}(g,l_n)\cap B_{\Ga_n}(h,l_n)$.
\end{itemize}
For each $n\in\IN$, we define $k_n:\Ga_n\times \Ga_n\to \IR$ by
$$k(g,h)=\left\{\begin{aligned}&\|t_g(g)(s(g))-t_g(h)(s(h))\|,&&\text{if }d(g,h)\leq l_n;\\&0,&&\text{otherwise}.\end{aligned}\right.$$
It is clear that $k$ is a kernel function conditionally of negative type in a $\frac{l_n}2$-bounded set, i.e., if $g_1,\cdots,g_n\in B(e,\frac{l_n}2)$ and $c_1,\cdots,c_n\in\IR$ with $\sum_{i=1}^nc_i=0$, one has that $\sum_{i,j=1}^nc_ic_jk(g_i,g_j)\leq 0$. That is because $k(g,h)=\|t_g(g)(s(g))-t_g(h)(s(h))\|=\|t_e(g)(s(g))-t_e(h)(s(h))\|$ by the second condition of fibred coarse embedding and $g\mapsto t_e(g)(s(g))$ forms a coarse embedding for $B(e,\frac{l_n}2)$. Define
$$\psi_n:\Ga_n\to\IR\quad\text{by}\quad \psi_n(g)=\frac{1}{\#\Ga_n}\sum_{h\in \Ga_n}k(h,hg).$$
Then $\psi_n$ is a function conditionally of negative type on $\Ga_n$ in a $\frac{l_n}2$-bounded set, indeed, it is an average of finite many kernel functions which are conditionally of negative type. Moreover, by definition, one can see that
$$|\psi_n(g)|\geq\rho_-(l_{\Ga_n}(g)),$$
this shows that $\psi$ is proper.

Now, define $\psi: \Ga^\infty_\omega\to\IR$ by $[(g_n)]\mapsto\lim_{n\to\omega}\psi_n(g_n)$. It is clear that this function is well-defined (it does not depend on the choice of the representation element). Moreover, since it is a limit of functions conditionally of negative type in larger and larger sets, it is also clear that $\psi$ is conditionally of negative type. Moreover, since $|\psi_n(g_n)|\geq\rho_-(l_{\Ga_n}(g_n))$ holds for all $n\in\IN$, passing to limit we also have that
$$|\psi([(g_n)])|\geq\rho_-(l_{\Ga_\omega^\infty}([(g_n)])),$$
this shows that $\Ga^\infty_\omega$ is a-T-menable.

\textbf{For (2)}. It is a direct corollary of Proposition \ref{pro: banach property T for limit group} if we take $\fB$ to be a Hilbert space and $p=2$.
\end{proof}

\begin{proof}[Proof of Theorem \ref{thm: FCE-by-FCE violates (T)}]
Assume for a contradiction that $(\Ga_n)_{n\in\IN}$ has geometric property (T). For any fixed ultrafilter $\omega\in\partial_{\beta}\IN$, by Lemma \ref{lem: exact sequence of limit groups}, we have the following short exact sequence
$$1\to N^{\infty}_{\omega}\to\Ga^{\infty}_{\omega}\to Q^{\infty}_{\omega}\to 1.$$
Since $(\Ga_n)_{n\in\IN}$ has geometric property (T), by Lemma \ref{lem: property for limit group}, the limit group $\Ga^\infty_\omega$ has property (T). Since $(\Ga_n)_{n\in\IN}$ has ``FCE-by-FCE'' structure, the limit groups $N^\infty_\omega$ and $Q^\infty_\omega$ should be a-T-menable. However, if $\Ga^\infty_\omega$ has property (T), then $(\Ga^\infty_\omega, N^\infty_\omega)$ and $Q^\infty_\omega$ should all have property (T), see \cite{KazhdanT}. Since $Q^\infty_\omega$ has both property (T) and Haagerup property, this implies that $Q^\infty_\omega$ is a finite group. Since $N^\infty_\omega$ is a-T-menable and $Q^\infty_\omega$ is finite, we can then conclude that $\Ga^\infty_\omega$ is a-T-menable. Now, we have that $\Ga^\infty_\omega$ has both property (T) and Haagerup property, this only happens if $(\Ga_n)_{n\in\IN}$ is a uniformly bounded sequence. This leads to a contradiction.
\end{proof}

\section{Coarse fixed point property}\label{sec: fixed point property}

In this section, we shall discuss a \emph{coarse fixed point property} for a sequence of Cayley graphs. The fixed point property for a bornological group is first introduced by R.~Tessera and J.~Winkel in \cite{TW2022}. In the same paper, they also provide a characterization of the geometric property
(T) for sequences of finite Cayley graphs regarding coarse fixed point property of a certain group. In this section, we shall generalize this result for geometric Banach property (T).

\subsection{Coarse fixed point property}

We shall first recall the definition of controlled action. Let $X=\bigsqcup_{n\in\IN}\Ga_n$ be the separated disjoint union of a sequence of finite Cayley graphs with uniform finite generators, $\fB$ a Banach space. Write $S_n\subseteq \Ga_n$ the generating set of $\Ga_n$ such that $\sup_{n\in\IN}\#S_n<\infty$. The set $S=\prod_{n\in\IN}S_n$ then forms a generating set the \emph{group of uniformly bounded product} $\prod^b_{n\in\IN}\Ga_n$. An isometric action of $\prod^b_{n\in\IN}\Ga_n$ on $\fB$ is a group homomorphism
$$\alpha: \prod^b_{n\in\IN}\Ga_n\to\text{Isom}(\fB),$$
where $\text{Isom}(\fB)$ is the group of all bijective isometries on $\fB$. By the Mazur-Ulam Theorem, the isometry group $\text{Isom}(\fB)$ can be decomposed into the semi-direct product
$$\text{Isom}(\fB)=O(\fB)\ltimes\fB.$$
Then the action $\alpha$ can also be split into two parts: the linear part $\pi: \prod^b_{n\in\IN}\Ga_n\to O(\fB)$ and a \emph{1-cocycle} $b:\prod^b_{n\in\IN}\Ga_n\to\fB$ associated with $\pi$, which means that
$$b(gh)=b(g)+\pi(g)b(h),$$
such that
$$\alpha(g)\xi=\pi(g)\xi+b(g)$$
for any $g,h\in\prod^b_{n\in\IN}\Ga_n$ and $\xi\in\fB$. This action $\alpha$ is \emph{controlled} if the corresponding cocycle $b$ is uniformly bounded on $S$. We shall also call the associated 1-cocycle $b$ a \emph{controlled 1-cocycle} if the affine action $\alpha$ is controlled.

\begin{Def}
With notations as above, the space $X$ is said to have \emph{coarse property $(F_{\fB})$} if for any controlled isometric action $\alpha: \prod^b_{n\in\IN}\Ga_n\to\text{Isom}(\fB)$, there exists a fix point for $\alpha$.
\end{Def}

Fix an isometric linear representation $\pi: \prod^b_{n\in\IN}\Ga_n\to O(\fB)$. Denote by $Z^1_{con}(\pi)$ the set of all controlled 1-cocycles associated with $\pi$. For any $\xi\in\fB$, one can construct a 1-cocycle $b_{\xi}$ associated with $\xi$ and $\pi$ by
$$b_{\xi}(g)=\xi-\pi(g)\xi.$$
Such a 1-cocycle is called a \emph{1-boundary} which is always controlled. The set of all 1-boundaries is denoted by $B^1(\pi)$, which is a subset of $Z^1_{con}(\pi)$. Notice that the affine isometric action $\alpha_{\xi}$ determined by $\pi$ and $b_{\xi}$ always has a fixed point, say $\xi$ itself. The proof of the following practical lemma can be found in \cite[Lemma 2.14]{BFGM2007}.

\begin{Lem}
Let $\fB$ be a uniformly convex Banach space, $\alpha: \prod^b_{n\in\IN}\Ga_n\to \text{Isom}(\fB)$ an affine isometric action. Let $\pi$ the linear part of $\alpha$, $b\in Z^1(\pi)$ the associated 1-cocycle. Then the following are equivalent:\begin{itemize}
\item[(1)]  all orbits of $\alpha$ is bounded;
\item[(2)] there exists a bounded orbit of $\alpha$;
\item[(3)] $\alpha$ has a fixed point;
\item[(4)] $b\in B^1(\pi)$.
\end{itemize}
As a corollary, $X$ has coarse property ($F_\fB$) if and only if for any isometric linear representation $\pi: \prod^b_{n\in\IN}\Ga_n\to O(\fB)$, the two set $Z^1_{con}(\pi)$ and $B^1(\pi)$ coincides.
\qed\end{Lem}

It is direct to see that the sum and scale multiple of a 1-cocycle still define a 1-cocycle, which makes $Z^1_{con}(\pi)$ a linear space. There is a canonical norm on $Z^1_{con}(\pi)$ defined as follow:
\begin{equation}\label{eq: norm of cocycle}\|b\|=\sup_{g\in S}\|b(g)\|_{\fB}.\end{equation}
To see this norm is well-defined, for any $b\in Z_{con}^1(\pi)$ with $\|b\|=0$, one then has that $b(g)=0$ for all $g\in S$. Since $S$ is the generating set of $\prod_{n\in\IN}^b\Ga_n$, any element in $\prod_{n\in\IN}^b\Ga_n$ can be written in the form $g_1g_2\cdots g_k$, where $g_1,g_2,\cdots,g_n\in S$. By definition, one can check that
$$b(g_1g_2\cdots g_k)=b(g_1g_2\cdots g_{k-1})+\pi(g_1g_2\cdots g_{k-1})b(g_k)=b(g_1g_2\cdots g_{k-1}).$$
By induction, we conclude that $b(g_1g_2\cdots g_k)=b(g_1)=0$. Thus, $b$ is the constant zero function. This prove that the positive definiteness of the norm. It is direct to check this norm is homogeneous and it satisfies the triangle inequality.

\begin{Lem}
With the notation stated as above, the space $Z^1(\pi)$ is a Banach space under the norm defined as in \eqref{eq: norm of cocycle}.
\end{Lem}

\begin{proof}
Let $\{b_n\}_{n\in\IN}$ be a Cauchy sequence in $Z^1(\pi)$. For any $g\in S$, since $\{b_n\}_{n\in\IN}$ is a Cauchy sequence, then $\{b_n(g)\}_{n\in\IN}$ defines a Cauchy sequence in $\fB$. Since $\fB$ is a Banach space, we define
$$b(g)=\lim_{n\to\infty}b_n(g).$$
Extend $b$ to a function on $\prod_{n\in\IN}^b\Ga_n$ under the rule
$$b(gh)=b(g)+\pi(g)b(h).$$
Since $S$ is the generating set of $\prod_{n\in\IN}^b\Ga_n$, thus this extension is well-defined. Then $b$ determines an element in $Z^1(\pi)$. Since the family of sequences $\{\{b_n(g)\}_{n\in\IN}\}_{g\in S}$ is a uniform Cauchy sequence, thus $b_n(g)$ uniformly converges to $b(g)$ on all $g\in S$. This proves that $\|b-b_n\|$ tends to $0$ as $n$ tends to infinity.
\end{proof}

We then have the following Delorme-Guicharde type theorem on coarse fixed point property and geometric property (T) for Banach space. However, we are only able to prove the implications of one of the directions.

\begin{Thm}[$(F_\fB)\Rightarrow (T_\fB)$]\label{thm: FB to TB}
Let $X=\bigsqcup_{n\in\IN}\Ga_n$ be the separated disjoint union of a sequence of finite Cayley graphs with uniformly finite generators, $\fB$ a uniformly convex Banach space. If $X$ has coarse property ($F_\fB$), then $X$ has geometric property ($T_\fB$).
\end{Thm}

\begin{proof}
Assume for a contradiction that $X$ does not have geometric property ($T_\fB$), i.e., there exists a representation $\pi:\IC_u[X]\to\L(\fB)$ which admits a $\varepsilon$-almost invariant vector for any $\varepsilon>0$. As we discussed above, this representation restricts to an isometric linear representation $\pi: \prod_{n\in\IN}^b\Ga_n\to O(\fB)$, a little abuse of notation, we still denote this group representation by $\pi$.

Define $\tau: \fB\to Z^1_{con}(\pi)$ by $\xi\mapsto b_\xi$ whose image is $B^1(\pi)$. Notice that this map is a bounded linear map with $\|\tau\|\leq 2$. Moreover, the kernel of this map is exactly the invariant subspace $\fB^{\pi}$. Assume that $B^1(\pi)$ is a close subspace in $Z^1_{con}(\pi)$. By open mapping theorem, the inverse map $\tau^{-1}: B^1(\pi)\to\fB/\fB^{\pi}$ also defines a bounded linear map. Then there exists $M>0$ such that
$$\|[\xi]\|\leq M\cdot\|\tau(\xi)\|=M\cdot\sup_{g\in S}\|\pi(g)\xi-\xi\|_{\fB}.$$
By Lemma \ref{lem: complemented representation}, this means the representation $\pi$ has a spectral gap which leads to a contradiction. Thus $B^1(\pi)$ is not close in $Z^1_{con}(\pi)$. Thus $B^1(\pi)\ne Z^1_{con}(\pi)$, which means that $X$ does not have coarse property ($F_\fB$).
\end{proof}

From the proof above, one can discern hints of defining \emph{uniform coarse property ($F_\fB$)}. If $X$ has coarse property ($F_\fB$), then for any representation $\pi$ of $\prod_{n\in\IN}^b\Ga_n$, the map $\tau_{\pi}:\fB\to Z^1_{con}(\pi)$ defined in the proof of Theorem \ref{thm: FB to TB} is a bounded linear surjection. We say $X$ has uniform coarse property ($F_\fB$) if $X$ has coarse property ($F_\fB$) and for any representation $\pi$, the family of inverse maps $\{\tau^{-1}_{\pi}\}$ is uniformly bounded. One can compare this with the uniform fixed point property for a group, for which the definition can be found in \cite[Remark 4.8]{DN2019}. Parallelly, we have the following result which can be proved similarly with Theorem \ref{thm: FB to TB}.

\begin{Cor}
Let $X=\bigsqcup_{n\in\IN}\Ga_n$ be the separated disjoint union of a sequence of finite Cayley graphs with uniform finite generators, $\fB$ a uniformly convex Banach space. If $X$ has uniform coarse property ($F_\fB$), then $X$ has uniform geometric property ($T_\fB$).\qed
\end{Cor}

It is worth mentionint that in \cite{MdlS2023}, A.~Marrakchi and M.~de~la Salle proved that any countable discrete group admits a proper affine isometric action on an $L^p$-space for sufficiently large $p$, which means any group can not have fixed point property for sufficiently large $p$ ($p$ could be $\infty$). It is appropriate to speculate that every space has a property ($T_{\sL ^ {\infty}}$). Now that we have proven Theorem \ref{thm: FB to TB}, it is natural to consider whether the implication holds in the opposite direction. Unfortunately, this is not the case. Parallelly thinking, the fixed point property for groups for a family Banach spaces is rather stronger than the associated property (T), see \cite{BFGM2007}. Therefore, it is reasonable to suspect that the reverse proposition is incorrect, even for a single Banach space. Indeed, one is referred to the following counterexample.

\begin{Exa}\label{exa: counter example of TB to FB}
Let $\fB=\IR$ be the $1$-dimensional Euclidean space with the canonical norm. Set $\Ga_n=\IZ_{2n+1}$. To understand a representation $\pi: \IC_u[X]\to\L(\IR)$, it suffices to clarify the orthogonal representations $\pi_O: \prod^b_{n\in\IN}\Ga_n\to O(\IR)$ and the representation of $\pi_{\ell^\infty}: \ell^2(X)\to\L(\IR)$. Since $O(\IR)=\IZ_2$ and the composition of $\pi_O$ and the canonical inclusion $\iota_n:\Ga_n\to\prod_{n\in\IN}^b\Ga_n$ by $\gamma\mapsto (e,\cdots,e,\gamma,e,\cdots)$ gives a group homomorphism from $\Ga_n=\IZ_{2n+1}$ to $\IZ_n$. Thus $\pi_O$ must be the unit map. Then the invariant space $\IR^{\pi}$ coincides with the entire space $\IR$. This proves that $X$ has geometric property $(T_\fB)$ directly from the definition.

However, $\prod^b_{n\in\IN}\Ga_n$ does not have coarse property ($F_\fB$). Notice that $X$ is actually a box space of $\IZ$. Thus $\IZ$ is a limit subgroup of $\prod^b_{n\in\IN}\Ga_n$ as we discussed in Section \ref{subsec: residually finite group}. Thus, there is a canonical group homomorphism from $\prod^b_{n\in\IN}\Ga_n\to\IZ$ as a quotient whenever we fix an ultrafilter. We shall define an isometric action $\alpha$ of $\prod^b_{n\in\IN}\Ga_n$ on $\IR$ by
$$\alpha_{z}:\xi\mapsto \xi+z$$
for any $z\in\IZ$, and it induces a controlled action of $\prod^b_{n\in\IN}\Ga_n$ via the quotient map. It is clear that this action is well-defined without fixed points.

Moreover, we should mention that this construction holds for any Banach space whose orthogonal group is $\IZ_2$. One can construct such a space of arbitrary dimensions even within the class of uniformly convex Banach space. For example, one can take a sufficiently asymmetric convex set in a Hilbert space and make it a unit ball for some norm of this Hilbert space. For such a space $\fB$, one can prove with a similar argument that geometric property ($T_\fB$) does not imply ($F_\fB$).
\end{Exa}

\subsection{Coarse property ($F_{\sL^p}$) and geometric property ($T$)}

In this subsection, we shall discuss the relationship between geometric property (T) and coarse fixed point property for $L^p$-spaces. The following theorem is the main result of this subsection.

\begin{Thm}\label{thm: TH to FLp}
Let $X$ be a separable disjoint union of a sequence of finite Cayley graphs. If $X$ has geometric property (T), then for any $p\in (1,2]$ and any subspace $\fB$ of any $L^p$-space $L^p(\mu)$, $X$ also has coarse property ($F_{\fB}$).
\end{Thm}

\begin{proof}
Let $X=\bigsqcup_{n\in\IN}\Ga_n$. Assume that $\alpha: \prod_{n\in\IN}^b\Ga_n\to\text{Isom}(\fB)$ to be a controlled isometric action. Define
$$\IH_s=\left\{\sum^m_ia_i\xi_i\,\Big|\, \xi_i\in\fB, a_i\in\IC\right\}.$$
For any $s>0$, we define a inner product on $\IH$ by
\begin{equation}\label{eq: inner product}\left\langle \sum^m_ia_i\xi_i, \sum^n_ja_j\eta_j\right\rangle_s=\sum_{i,j}a_i\overline{b_j}e^{-s\|\xi_i-\eta_j\|^p}.\end{equation}
By \cite{BDK1965} (or \cite[Theorem 5.1]{WW1975}), one has that $(\xi,\eta)\mapsto\|\xi-\eta\|^{p}$ is conditionally of negative type when $p\in(1,2]$. By Schoenberg's Theorem, see \cite{Schoenberg1938}, the inner product defined as in \eqref{eq: inner product} is positive definite, thus making $\IH$ a pre-Hilbert space. Denote by $\H_s$ the completion of $\IH$ under the inner product $\langle,\rangle_s$. We then define
$$\iota_s:\fB\to\H_s\quad\text{by}\quad \xi\mapsto \xi\in\IH\subseteq\H_s.$$
Notice that the image of $\iota_s$ spans $\IH$. Moreover, we define $\rho_s:\prod_{n\in\IN}^b\Ga_n\to U(\H_s)$ by
$$\rho_s(\gamma)\iota_s(\xi)=\iota_s(\alpha(\gamma)(\xi)).$$
Since the linear span of the image of $\iota_s$ is dense in $\H_s$, the unitary representation $\rho_s$ is well-defined. 

Let $\phi_n:\ell^{\infty}(X)\to\IC$ be the state we defined in Section \ref{subsec: Banach T for limit groups}. Choose $\phi$ to be any cluster point of the sequence $\{\phi_n\}_{n\in\IN}$, which defines a $\prod_{n\in\IN}^b\Ga_n$-invariant positive functional on $\ell^{\infty}(X)$. By Riesz representation theorem, we denote by $\mu_{\phi}$ the corresponding $\prod_{n\in\IN}^b\Ga_n$-invariant measure $\mu_{\phi}$ on ${\beta}X$. We denote $\lambda$ the left-regular representation of $\prod_{n\in\IN}^b\Ga_n$ on $L^2(\mu_\phi)$ and $\varphi$ the multiplication representation of $\ell^\infty(X)$ on $L^2(\mu_\phi)$. Consider the representation $\pi_s:\IC_u[X]\to\L(L^2(\mu_\phi)\ox\H_s)$, defined by
$$\pi_s(\gamma)=\lambda\ox\rho_s(\gamma)\quad\text{for any }\gamma\in\prod_{n\in\IN}^b\Ga_n,\quad\text{and}\quad \pi_s(f)=\varphi\ox 1(f)\quad\text{for any }f\in\ell^\infty(X).$$
It is easy to check that this representation is well-defined. 

We now claim that $\rho_s$ must have an invariant vector in $\H_s$ for sufficiently large $s$. Fix $\xi\in\fB$. Since the action $\alpha$ is controlled, thus $S\xi$ is bounded in $\fB$, where $S\subseteq\prod^b_{n\in\IN}\Ga_n$ is the generating set. We denote
$$R_0=\sup_{\gamma\in S}\|\gamma\xi-\xi\|<\infty.$$
Thus we have that
$$\inf_{\gamma\in S}|\langle\pi_s(\gamma)\iota_s(\xi),\iota_s(\xi)\rangle_s|=e^{-s\|\alpha(\gamma)\xi-\xi\|^p}\geq e^{-sR_0^p}\to 1\quad\text{as }s\to 0.$$
For any $\varepsilon>0$, one can find an $\varepsilon$-invariant vector for $\rho_s$ when $s$ is sufficiently small. Assume for a contradiction that $\rho_s$ has no non-zero invariant vectors for any $s>0$. Then $\H_s$ is isomorphic to the quotient space $(1\ox\H_s+L^2(\mu_\phi,\H_s)^{\pi_s})/L^2(\mu_\phi,\H_s)^{\pi_s}$. By the fact that $X$ has geometric property (T), the representation $\pi_s$ should have a spectral gap. It leads to a contradiction that $\pi_s$ has an $\varepsilon$-invariant vector when $s$ is sufficiently small.

Say $v$ is the invariant vector of $\rho_s$. Fix $\varepsilon_0>0$, then there exists $\xi_0\in\fB$ such that $\|\iota_s(\xi_0)-v\|\leq\varepsilon_0$. Then for any $\gamma\in\prod_{n\in\IN}^b\Ga_n$, we conclude that
$$\|\rho_s(\gamma)\iota_s(\xi_0)-\iota_s(\xi_0)\|\leq\|\rho_s(\gamma)\iota_s(\xi_0)-\rho_s(\gamma)v\|+\|v-\iota_s(\xi_0)\|\leq 2\varepsilon_0.$$
By definition, we have that
$$\langle\rho_s(\gamma)\iota_s(\xi_0),\iota_s(\xi_0)\rangle_s=e^{-s\|\alpha(\gamma)\xi_0-\xi_0\|^p}\geq 1-2\varepsilon^2_0.$$
Thus, the orbit of $\xi_0$ must be bounded. This means that $X$ has coarse property ($T_{\fB}$).
\end{proof}

We should mention that the proof of the above theorem is a modified Delorme–Guichardet argument for $(T)\Rightarrow (F_\H)$. As a direct corollary, we have the following result which compares the two closely related concepts of Geometric property ($T$) and Geometric property ($T_{L^p(\mu)}$). For $p\geq 2$, Theorem \ref{thm: TH to FLp} may not hold. For example, G.~Yu proved in \cite{Yu2005} that hyperbolic groups admit a proper group action on $\ell^p$ for sufficiently large $p$. Fix a residually finite, hyperbolic group $\Ga$ with property (T), then the box space of $\Ga$ does not have coarse property $(F_{\ell^p})$ using the same construction in Example \ref{exa: counter example of TB to FB} as $\Ga$ is the limit space of this box space. However, by Theorem \ref{thm: box iff group in lp}, such box space has geometric property ($T$). 

\begin{Thm}\label{thm: lp vs l2}
Let $X$ be a separable disjoint union of a sequence of Cayley graphs. If $X$ has geometric property (T), then for any $p\in (1,\infty)\backslash\{2\}$, $X$ also has geometric property ($T_{L^p(\mu)}$) for any $L^p$-space $L^p(\mu)$. Moreover, $X$ has uniform geometric property $(T_{\sL^p})$.
\end{Thm}

\begin{proof}
Since $X$ has geometric property (T), by Theorem \ref{thm: TH to FLp} and Theorem \ref{thm: FB to TB}, we conclude that $X$ has geometric property ($T_{L^p(\mu)}$) for any $p\in(1,2]$. By Theorem \ref{thm: TB to TB*}, we conclude that $X$ also has geometric property ($T_{L^p(\mu)}$) for any $p\in[2,\infty)$. By Proposition \ref{pro: uniform Tlp and Tlp}, we conclude that $X$ has uniform geometric property ($T_{\sL^p}$).
\end{proof}

\subsection{Geometric property ($T_{\sL^p}$) vs. Geometric property (T)}

In this subsection, we shall provide another approach to Theorem \ref{thm: lp vs l2} without using the fixed point property.

Before we prove this theorem, we shall first recall some basic facts on $L^p$-spaces. For any $p,q\in[1,\infty)$, one can define the \emph{Mazur map} $M_{p,q}:S(L^p(\mu))\to S(L^q(\mu))$ by
$$M_{p,q}(f)=sign(f)\cdot|f|^{p/q}.$$
It is direct to check that $M_{p,q}$ and $M_{q,p}$ are inverse to each other. For any $p\in (1,\infty)\backslash\{2\}$, let $V$ be an isometric linear map on $L^p(\mu)$. Then the conjugation $M_{p,2}\circ V\circ M_{2,p}$ can extends to an isometric linear map on $L^2(\mu)$. It is not hard to see this linear map is isometric. To see it is linear, one should need the celebrated theorem proved by Banach for $([0,1],\lambda)$ and generalized by Lamperti for $\sigma$-finite measure space, see \cite[Theorem 3.1]{Lamperti}.

\begin{Thm}[Banach/Lamperti]
Let $p\in (1,\infty)\backslash\{2\}$, $V$ an isometric linear map on $L^p(Y,\mu)$. Then
$$(Vf)(x)=f(T_V(x))\cdot h(x)\cdot \left(\frac{dT_*\mu}{d\mu}(x) \right),$$
where $T_V: Y\to Y$ is a measurable, regular set isomorphism, $h$ is a measurable function on $Y$ such that $|h|=1$ almost everywhere.
\end{Thm}

By Banach-Lamperti Theorem, it is direct to check that $M_{p,2}\circ V\circ M_{2,p}$ is linear. The essential reason is that $V$ is essentially determined by a measurable transformation on the base space $Y$, while the Mazur map only changes the value of a function. Now, we are ready to prove Theorem \ref{thm: lp vs l2}. Actually, a similar result also holds for noncommutative $L^p$-spaces, see \cite{Olivier2012} for example.

\begin{proof}[Proof of Theorem \ref{thm: lp vs l2}]
Write $X=\bigsqcup_{n\in\IN}\Ga_n$. Assume for a contradiction that $X$ does not have geometric property ($T_{L^p(\mu)}$), i.e., there exists a representation $\pi:\IC_u[X]\to\L(L^p(\mu))$ such that one can choose a sequence of vectors $\{\xi_n\}_{n\in\IN}$ in $L^p(\mu)$ such that $d(\xi, L^p(\mu)^\pi)=1$ and
$$\lim_{n\to\infty}\sup_{V\in\prod_{n\in\IN}S_n\subseteq\IC_u[X]}\|\pi(V)\xi_n-\xi_n\|=0.$$
To clarify, we shall write the measure space as $(Y,\mu)$.

Restrict this representation $\pi$ on $\ell^{\infty}(X)$, it gives a contractive representation of $\ell^{\infty}(X)$ on $L^p(\mu)$. By \cite[Theorem 4.5]{PV2020}, we conclude that $\pi(\ell^{\infty}(X))$ lies in $L^\infty(Y,\mu)$. Then it induces a representation $\wt\pi:\IC_u[X]\to L^2(\mu)$ as follows. Since every element $T\in\IC_u[X]$ can be written as a finite sum of $f_{\gamma}\cdot\gamma$, where $\gamma=\bigoplus_{n\in\IN}\gamma_n\in\prod^b_{n\in\IN}\Ga_n$ and $f_\gamma\in\ell^{\infty}(X)$. Thus it suffices to clarify how $\ell^{\infty}(X)$ and $\prod^b_{n\in\IN}\Ga_n$ act on $L^2(\mu)$. For $\ell^{\infty}(X)$, we define $\wt\pi(f)=\pi(f)\in L^{\infty}(Y,\mu)\subseteq\B(L^2(\mu))$ for any $f\in\ell^\infty(X)$. For any $\gamma\in\prod_{n\in\IN}^b\Ga_n$, we define
$$\wt\pi(\gamma)=M_{p,2}\circ\pi(\gamma)\circ M_{2,p},$$
where $M_{2,p}: S(L^2(\mu))\to S(L^p(\mu))$ and $M_{p,2}: S(L^p(\mu))\to S(L^2(\mu))$ are the Mazur maps. From the discussion above, we have that $\wt\pi$ forms a unitary representation of $\prod^b_{n\in\IN}\Ga_n$ on $L^2(\mu)$. Thus $\wt\pi: \IC_u[X]\to\L(L^p(\mu))$ provides a $*$-representation on the Hilbert space $L^2(\mu)$.

Notice that $\xi\in L^p(\mu)^\pi$ if and only if $M_{p,2}(\xi)$ is in $L^2(\mu)^{\wt\pi}$ by definition. Denote $\eta_n=M_{p,2}(\xi_n)$ for each $n\in\IN$. Since the Mazur map is uniformly continuous, see \cite[Theorem 9.1]{BL2000}, we conclude that $d(\eta_n,L^2(\mu)^{\wt\pi})>\delta$ for some $\delta>0$. Moreover, one can check that for any $V=\bigoplus s^{(n)}\in\prod_{n\in\IN}S_n\subseteq\IC_u[X]$ where each $s^{(n)}\in S_n$ is a generator, 
\[\|\wt\pi(V)\eta_n-\eta_n\|=\|M_{p,2}(\pi(V)\xi_n)-M_{p,2}(\xi_n)\|.\]
By the uniform continuity of the Mazur map and the fact that $\|\pi(V)\xi_n-\xi_n\|$ is sufficiently small for sufficiently large $n\in\IN$, we conclude that
$$\lim_{n\to\infty}\sup_{V\in\prod_{n\in\IN}S_n\subseteq\IC_u[X]}\|\wt\pi(V)\eta_n-\eta_n\|=0.$$
This leads to a contradiction that $X$ has geometric property (T). The last claim holds as a direct corollary of Proposition \ref{pro: uniform Tlp and Tlp}.
\end{proof}

\section{Coarse invariance}\label{sec: Coarse invariance}

In this section, we shall prove that the geometric property ($T_{\sB}$) is a coarse invariant. The main result of this section is the following.

\begin{Thm}\label{thm: coarse invariance}
Let $\sB$ be a uniformly convex family of Banach spaces which is close under taking subspaces and finite direct sums, and let $X, Y$ be spaces. If $X$ is coarsely equivalent to $Y$, then $X$ has geometric property ($T_\sB$) if and only if $Y$ has geometric property ($T_\sB$).
\end{Thm}

Let $\fB_1,\cdots,\fB_n$ be Banach spaces, the direct sum of $\fB_1,\cdots,\fB_n$, denoted by $\fB_1\oplus\cdots\oplus\fB_n$, is equipped with the canonical norm
$$\|(\xi_1,\cdots,\xi_n)\|^2=\sum_{i=1}^n\|\xi_i\|^2_{\fB_i}.$$
We shall denote $\fB^n=\fB\oplus\cdots\oplus\fB$ the $n$-direct sum of $\fB$.

Let $f: X\to Y$ be an injective coarse equivalence. By \cite[Lemma 4.1]{WY2014}, one can find a net $X'\subseteq X$ and a net $Y'\subseteq Y$ such that $f|_{X'}$ is a bijection between $X'$ and $Y'$. A bijective coarse equivalence $f|_{X'}$ will induce an isomorphism between $\IC_u[X']$ and $\IC_u[Y']$. Thus, it suffices to prove that the case for $f$ is an injective coarse equivalence. For any $N\in\IN$, we shall treat $N$ as the set $\{1,\cdots, N\}$ with a bit of abuse of notation. Since $X$ and $Y$ both have bounded geometry, there exists $N\in\IN$ such that the first section inclusion $\iota: X\to X\times N$ given by $x\mapsto (x,1)$ factors through $f: X\to Y$ by injective coarse equivalence, i.e., there exists an injective coarse equivalence $g: Y\to X\times N$ such that $\iota=g\circ f$. These maps induce injective homomorphisms on the level Roe algebras as follow
$$\iota_*: \IC_u[X]\xrightarrow{f_*}\IC_u[Y]\xrightarrow{g_*}\IC_u[X\times N]\cong M_N(\IC_u[X]).$$
For any representation $\pi_X: \IC_u[X]\to\L(\fB_X)$, it induces a canonical representation $\pi^N_X: \IC_u[X\times N]\to\L(\fB^N_X)$ by matrix action. Since $Y$ is a subspace of $X\times N$, we shall denote $\fB_Y$ to be the range of $\pi^N_X(\chi_Y)$. Then $\pi_X^N$ restricts to a representation of $Y$
$$\pi_Y:\IC_u[Y]\to\L(\fB_Y),$$
and $\pi_X=\pi_Y\circ f_*$.

\begin{Lem}\label{lem: from X to Y}
With notation as above, $\fB^{\pi_X}_X=\pi_Y(\chi_X)(\fB_Y^{\pi_Y})$. As a result, the map $\chi_X: \fB_Y\to\fB_X$ descends to an operator from $\fB_X/\fB^{\pi_X}_X$ onto $\fB_Y/\fB_Y^{\pi_Y}$.
\end{Lem}

\begin{proof}
From the definition, it is direct to see that $\pi_Y(\chi_X)(\fB_Y^{\pi_Y})\subseteq \fB_X^{\pi_X}$ since $\pi_X=\pi_Y\circ f_*$. To see $\fB_X^{\pi_X}\subseteq \pi_Y(\chi_X)(\fB_Y^{\pi_Y})$, we still need the following claim.

\noindent{\bf Claim.} $(\fB^N_X)^{\pi_X^N}=\{(\xi,\cdots,\xi)\in\fB^N_X\mid \xi\in\fB^{\pi_X}_X\}$.

\noindent{\emph{Proof of the Claim.}} Let $(\xi_1,\cdots,\xi_N)\in (\fB^N_X)^{\pi_X^N}$. By Lemma \ref{lem: invariant vector in group language}, it suffices to consider full partial translations. Notice that the right-shift operator in $M_N(\IC)$ determines a full partial translation in $M_N(\IC_u[X])$ by viewing the matrix entries as the identity of $\IC_u[X]$. Thus, we have that
$$(\xi_1,\cdots,\xi_n)=(\xi_n,\xi_1,\cdots,\xi_{n-1}).$$
As a result, $\xi_1=\xi_2=\cdots=\xi_n=\xi$. Moreover, the representation of $M_N(\IC_u[X])$ on $\fB^N_X$ is equal to $\pi_X$ when restricts to $\IC_u[X]$. Thus, $\xi\in \fB^{\pi_X}_X$. This proves the left side is contained in the right side. The other side of containment holds directly from the definition.

Notice that $\pi_X^N(\chi_Y)((\fB^N_X)^{\pi^N_X})\subseteq \fB_Y^{\pi_Y}$ by definition. Apply both sides by $\chi_X$, we have that
$$\pi_X^N(\chi_X)((\fB^N_X)^{\pi^N_X})\subseteq \pi_X^N(\chi_X)(\fB_Y^{\pi_Y}),$$
i.e.,
$$\fB^{\pi_X}_X\subseteq \pi_Y(\chi_X)(\fB_Y^{\pi_Y}).$$
This finishes the proof.
\end{proof}

On the other hand, since $f: X \to Y$ is an injective coarse equivalent, we shall identify $X$ as a net of $Y$. Take $\{N_x\}_{x\in X}$ to be a uniformly bounded disjoint cover of $Y$ such that $x\in N_x$ for any $x\in X$. There exists $N\in\IN$ such that $\sup_{x\in X}\#N_x\leq N$. Label each element in $N_x$ by $\{1,\cdots,\#N_x\}$, then $N_x=\{x(1),\cdots,x({\#N_x})\}$ where $x=x_1$. We shall denote $Y^{(n)}=\{x(n)\}_{x\in X}$, then $Y=\bigsqcup_{i=1}^N Y^{(i)}$ and $X= Y^{(1)}$.

For any representation $\pi_Y: \IC_u[Y] \to \L(\fB)$, we shall denote $\fB^{(n)}=\chi_{Y^{(n)}}\fB$. Then $\fB=\bigoplus_{n=1}^N\fB^{(n)}$ and $\pi_Y$ induces a representation $\pi_X: \IC_u[X]\to \L(\fB^{(1)})$. For each $n\in  N$, we shall denote
$$\iota^{(n)}: Y^{(n)}\to X, \quad x(n)\mapsto x(1).$$
Notice that $\iota^{(n)}$ is a partial translation from $Y^{(n)}$ to $X^{(n)}=Im(\iota^{(n)})$. We shall denote $V^{(n)}$ the associated partial translation operator. Inspired by Lemma \ref{lem: from X to Y}, we have the following result.

\begin{Lem}\label{lem: from Y to X}
An vector $\xi=(\xi_1,\cdots,\xi_N)\in \fB$ is $\pi_Y$ invariant if and only if $\xi_1\in(\fB^{(1)})^{\pi_X}$ and $\pi_Y(V^{(n)})\xi_n=\pi_X(\chi_{X^{(n)}})\xi_1$.
\end{Lem}

\begin{proof}
To see the $(\Rightarrow)$ part, for any full partial translation in $V\in \IC_u[X]\subseteq \IC_u[Y]$, we have that
$$\pi_Y(V)\xi=\pi_X(V)\xi_1=\xi_1.$$
This means that $\xi_1\in(\fB^{(1)})^{\pi_X}$. Consider $(V^{(n)})^*$ to be the conjugate of $V^{(n)}$, it is direct to see that
$$(V^{(n)})^*V^{(n)}=\chi_{Y^{(n)}},\quad V^{(n)}(V^{(n)})^*=\chi_{X^{(n)}}.$$
Then we have that
$$\xi_n=\chi_{Y^{(n)}}\xi=\pi_Y((V^{(n)})^*)\xi=\pi_Y((V^{(n)})^*)\xi_1.$$
Apply both sides with $\pi_Y(V^{(n)})$, we conclude that $\pi_Y(V^{(n)})\xi_n=\pi_X(\chi_{X^{(n)}})\xi_1$.

For the $(\Leftarrow)$ part, take $V\in \IC_u[Y]$ to be a full partial translation. Denote by $V_{mn}=\chi_{Y^{(m)}}V\chi_{Y^{(n)}}$. To show such $\xi$ as claimed is invariant, it suffices to prove that
$$\xi_m=\sum_{n=1}^NV_{mn}\xi_n.$$
Notice that $\xi_n=(V^{(n)})^*\xi_1$ and $V^{(m)}V_{mn}(V^{(n)})^*$ is a partial translation in $\IC_u[X]$. Moreover, since $V$ is a full partial translation, we can denote the associated bijection by $t_V: Y\to Y$. Then $V_{mn}=\chi_{Y^{(m)}\cap t_V(Y^{(n)})}V$ and 
$$t_V(Y^{(n)})\cap t_V(Y^{(k)})=\emptyset$$
whenever $n\ne k$. As a result,
$$\sum_{n=1}^NV^{(m)}V_{mn}=\sum_{n=1}^NV^{(m)}\chi_{Y^{(m)}\cap t_V(Y^{(n)})}V=V^{(m)}\chi_{Y^{(m)}}V=\chi_{X^{(m)}}V^{(m)}V.$$
Then 
$$V^{(m)}\sum_{n=1}^NV_{mn}\xi_n=\sum_{n=1}^NV^{(m)}V_{mn}(V^{(n)})^*\xi_1.$$
Since $\xi_1$ is invariant by definition, $V^{(m)}V_{mn}(V^{(n)})^*$ is a partial translation in $\IC_u[X]$ and the ranges of $\{V^{(m)}V_{mn}(V^{(n)})^*\}$ are disjoint for different $n$, we conclude that
$$\sum_{n=1}^NV^{(m)}V_{mn}(V^{(n)})^*\xi_1=\chi_{X^{(m)}}\xi_1.$$
Apply both sides with $(V^{(m)})^*$, we have that
$$\xi_n=(V^{(n)})^*\chi_{X^{(m)}}\xi_1=(V^{(m)})^*V^{(m)}\sum_{n=1}^NV_{mn}\xi_n=\sum_{n=1}^N\chi_{Y^{(m)}}V_{mn}\xi_n=\sum_{n=1}^NV_{mn}\xi_n.$$
This finishes the proof.
\end{proof}

With a similar argument as in Lemma \ref{lem: from X to Y}, there is a direct corollary of Lemma \ref{lem: from Y to X} as follows.

\begin{Cor}\label{cor: from Y to X}
For any representation $\pi_Y: \IC_u[Y] \to \L(\fB)$, the invariant space $(\chi_X\fB)^{\pi_X}$ for the induced representation $\pi_X:\IC_u[X]\to\L(\chi_X\fB)$ is equal to $\chi_X(\fB^{\pi_Y})$. As a result, $\chi_X$ descends to a quotient map $\fB/\fB^{\pi_Y}\to \chi_X\fB/(\chi_X\fB)^{\pi_X}$.
\end{Cor}

Now, we are ready to prove Theorem \ref{thm: coarse invariance}.

\begin{proof}[Proof of Theorem \ref{thm: coarse invariance}]
For the ($\Rightarrow$) part, assume that $X$ has geometric property ($T_{\sB}$). For any $\fB\in\sB$ and any representation $\pi_Y: \IC_u[Y]\to\L(\fB)$, it induces a representation $\pi_X:\IC_u[X]\to\L(\chi_X\fB)$ as above. Then $\pi_X$ has a spectral gap. Assume for a contradiction that $\pi_Y$ has no spectral gap, i.e., there exists a sequence $[\xi_n]\in\fB/\fB^{\pi_Y}$ such that for any full partial translation $V\in \IC_u[Y]$ with $\supp(V)\subseteq E^Y_0$, one has that $\|\pi(V)[\xi_n]-[\xi_n]\|\leq \frac 1n$. By Corollary \ref{cor: from Y to X}, this sequence determines a sequence of vectors in $\fB_X/\fB_X^{\pi_X}$ via $\chi_X$. Since any full partial translation in $\IC_u[X]$ extends to a full partial translation in $\IC_u[Y]$ by adding the identity operator on $X^c$, it is direct to see that the existence of the sequence $\{\chi_X([\xi_n])\}$ leads to a contradiction that $\pi_X$ has a spectral gap. This finishes the proof of the ($\Rightarrow$) part.

For the ($\Leftarrow$) part, assume that $Y$ has geometric property ($T_{\sB}$). Let $N\in\IN$ be such that
$$\iota: X\xrightarrow{f}Y\xrightarrow{g}X\times N$$
the first section inclusion $\iota: X\to X\times N$ given by $x\mapsto (x,1)$ factors through $f: X\to Y$ by injective coarse equivalence. As we have proved the ($\Rightarrow$) part, it is direct that $X\times N$ have geometric property ($T_{\sB}$). For any representation $\pi_X:\IC_u[X]\to\L(\fB)$, we shall consider the canonical representation
$$\pi_X^N: M_N(\IC_u[X])\to\L(\fB^N).$$
Assume for a contradiction that $\pi_X$ has no spectral gap, i.e., there exists a sequence $[\xi_n]\in\fB/\fB^{\pi_X}$ such that for any full partial translation $V\in \IC_u[X]$ with $\supp(V)\subseteq E^X_0$, one has that $\|\pi(V)[\xi_n]-[\xi_n]\|\leq \frac 1n$. As a corollary of Lemma \ref{lem: from X to Y}, the vector $[\eta_n]=[(\xi_n,\cdots,\xi_n)]$ determines a uniformly bounded vectors in $\fB^N_X/(\fB^N_X)^{\pi_X^N}$. For any full partial translation in $M_N(\IC_u[X])$, we can write it as $V=(V_{ij})_{i,j=1,\cdots,N}$, where each $V_{ij}$ is a partial translation. With a similar argument as Lemma \ref{lem: from Y to X}, one can similarly prove that the sequence $\{\eta_n\}$ satisfies that $\|V[\eta_n]-[\eta_n]\|$ tends to $0$ as $n$ tends to infinity since $\|\sum_{j}V_{ij}\xi_n-\xi_n\|\leq\sum_{j}\|V_{ij}\xi_n-\Phi(V_{ij})\xi_n\|$ tends to $0$ as $n$ tends to infinity. This leads to a contradiction with $X\times N$ having geometric property ($T_{\sB}$). We then finish the proof.
\end{proof}

\section{Open questions}\label{sec: Open questions}

In this last section, we list several interesting open problems.

\begin{Que}
Can the geometric property $(T_{\sL^p})$ provide a counterexample for the maximal $L^p$-coarse Baum Connes conjecture for $p\in(1,\infty)$ ?
\end{Que}

Let $X$ be a metric space with bounded geometry, and $Z$ a countable dense subset of $X$. For $p\in[1,\infty)$, recall that an operator $T\in\B(\ell^p(Z) \ox \ell^p(\IN))$ has finite propagation if there exists $R>0$ such that $\chi_V T\chi_U=0$ whenever $d(U,V)\geq R$. Here we view $\chi_{V}, \chi_U$ as multiplication operators on $\ell^p(Z) \ox \ell^p(\IN)$. The operator $T$ is called locally compact if $\chi_KT$ and $T\chi_K$ are in $\K(\ell^p(Z) \ox \ell^p(\IN))$ for any bounded subset $K$. The algebraic $L^p$-Roe algebra, denoted by $\IC^p[X]$, is defined to be the set of all bounded linear operators on $\ell^p(Z)\ox\ell^p(\IN)$ which is locally compact and has finite propagation. Its completion in $\B(\ell^p(Z) \ox \ell^p(\IN))$ is denoted by $C^p(X)$, which is called the \emph{(reduced) $L^p$-Roe algebra}. One is referred to \cite{ZhangZhou2021} for details of the $L^p$-Roe algebra and the $L^p$-coarse Baum-Connes conjecture.

Analogue to the $L^2$ case, we have the following two definitions for the maximal $L^p$-coarse Baum-Connes conjecture. The first one is to consider the completion of $\IC^p[X]$ under the norm
\[\|T\|_{\sL^p,\max}=\sup\{\|\pi(T)\|\mid \pi:\IC^p[X]\to\L(L^p(\mu))\text{ is a representation}, \ L^p(\mu)\in\sL^p\}\]
which is similar to Definition \ref{def: maximal norm for family of representations}. The other one is to consider the norm
\[\|T\|_{\max,\ell^p}=\sup\{\|\pi(T)\|\mid \pi:\IC^p[X]\to\L(\ell^p)\text{ is a representation}\}.\]
The completion of $\IC^p[X]$ under these two norms are denoted by $C^p_{\max,\sL^p}(X)$ and $C^p_{\max,\ell^p}(X)$, respectively.
By Yu's localization technique, one can define the localization algebra via the maximal versions of the $L^p$-Roe algebras, denoted by $C^p_{L,\max,\sL^p}(X)$ and $C^p_{L,\max,\ell^p}(X)$. The evaluation map from the localization algebra of the Rips complex to the Roe algebra of the Rips complex induces a canonical assembly map as the scale of the Rips complex tends to infinity. The maximal $L^p$-coarse Baum-Connes conjecture can be defined by claiming the assembly map is an isomorphism.

Notice that $L^p(0,1)$ is not isomorphic to $\ell^p$ and it is unknown to us whether the compact operator algebra $\K(L^p(\mu))$ is isomorphic to that of $\K(\ell^p)$. Thus, it is natural to ask whether the two completions above coincide, or have the same $K$-theory? And which one should be proper to define the maximal $L^p$-coarse Baum-Connes conjecture? By Yu's cutting and pasting argument \cite{Yu1997}, can one prove that the $K$-theory of these two maximal $L^p$-versions of localization algebras is isomorphic to the $K$-homology group? As we proved in Section \ref{sec: Kazhdan Idempotent}, when $p\in(1,\infty)$, one can find Kazhdan projections in both $C^p_{\max,\sL^p}(X)$ and $C^p_{\max,\ell^p}(X)$ for spaces with geometric property ($T_{\sL^p}$). It is proved in \cite{ChungNowak2023} that expanders are counterexamples to the $L^p$-coarse Baum-Connes conjecture. Then, can the existence of Kazhdan projections provide counterexamples to the maximal version of the $L^p$-coarse Baum-Connes conjecture, as in \cite{WilYu2012b}?

\begin{Que}
Can one prove the $K$-amenability for the $L^p$-Roe algebras for spaces that admit a coarse embedding into an $\ell^q$-space?
\end{Que}

Building upon the discussion in the preceding question, we may further investigate the $L^p$-analogue of K-amenability. The $K$-amenability of a Roe algebra implies that the maximal and reduced completions of the uniform Roe algebra share identical K-theory. In \cite{SpaWil2013}, J. Špakula and R. Willett proved that metric spaces with bounded geometry admitting coarse embeddings into Hilbert spaces are $K$-amenable. Their methodology relies crucially on the geometric Dirac-dual-Dirac method pioneered by Yu in \cite{Yu2000}.

Recently, J. Wang, Z. Xie, G. Yu, and B. Zhu proved in \cite{WXYZ2024} that if a bounded geometry metric space admits a coarse embedding into an $\ell^p$-space, then the $L^q$-coarse Baum-Connes conjecture holds for such spaces. This proof similarly employs a Dirac-dual-Dirac-type approach. This naturally raises the following question: For bounded geometry metric spaces that admit coarse embeddings into $\ell^p$-spaces, do they satisfy the $L^q$-version of $K$-amenability, i.e.,
$$K_*(C^p(X))\cong K_*(C^p_{\max,\sL^p}(X))\cong K_*(C^p_{\max,\ell^p}(X)) ?$$

\begin{Que}
Is FCE$_{\fB}$ (or {FCE$_{\fB}$-by-FCE$_{\fB}$}) compatible with geometric property ($T_\fB$)?
\end{Que}

This question is a Banach analogue of Theorem \ref{thm: FCE-by-FCE violates (T)}. .We say a sequence of finite groups has FCE$_\fB$ if it admits a fibred coarse embedding into $\fB$. If $\fB$ is an $L^p$-space and the sequence of finite groups forms a box space of a residually finite group, then the limit group admits a proper isometric action on some $L^p$-space. This proof employs the ultraproduct construction for $L^p$-spaces, and we believe it can be extended to general limit groups. However, when attempting to replicate Theorem \ref{thm: FCE-by-FCE violates (T)}, one should notice that our proof fundamentally relies on the Delorme-Guichardet theorem, i.e., the fixed-point characterization of Property (T). Regrettably, for general Banach spaces, Property ($T_{\fB}$) lacks an analogous fixed-point characterization. This raises a question: for $\fB$ an $L^p$-space (even an arbitrary uniformly convex Banach space), are \emph{FCE$_{\fB}$} and property ($T_{\fB}$) mutually compatible?

\bibliographystyle{alpha}
\bibliography{ref}
\end{document}